\documentclass[12pt,reqno,a4paper]{amsart}
\usepackage{amsmath,amssymb,amsfonts,amscd}
\usepackage[mathscr]{eucal}
\usepackage{stackrel}
\usepackage{comment,textcomp}
\usepackage[usenames]{color}
\usepackage{mathtools}

\usepackage{hyperref}  

\date{30 December   2023 (12 January 2024)}  

\author{Theodore~Th. Voronov}
\address{Department of Mathematics,  University of Manchester,    Manchester,   M13 9PL,  UK}
\email{theodore.voronov@manchester.ac.uk}
\thanks{}
\dedicatory{Dedicated to the memory of Kirill Mackenzie (1951--2020)}

\title[Tangent functor on microformal morphisms]{Tangent functor on microformal morphisms, and non-linear pullbacks for forms and cohomology}

\newtheorem{theorem}{Theorem}

\newtheorem{proposition}{Proposition}
\newtheorem{corollary}{Corollary}

\theoremstyle{definition}
\newtheorem{definition}{Definition}

\newtheorem{example}{Example}
\newtheorem{remark}{Remark}

\def\co{\colon\thinspace}

\renewcommand{\geq}{\geqslant}

\DeclareMathOperator{\id}{id}

\DeclareMathOperator{\fun}{\mathit{C^{\infty}}}
\DeclareMathOperator{\funn}{\mathbf{C^{\infty}}}

\DeclareMathOperator{\ofunn}{\mathbf{\Pi\!C^{\infty}}}

\DeclareMathOperator{\omm}{\mathbf{\Omega}}
\DeclareMathOperator{\pomm}{\mathbf{\Pi \Omega}}

\DeclareMathOperator{\HH}{\mathbf{H}}
\DeclareMathOperator{\PH}{\mathbf{\Pi H}}

\newcommand{\der}[2]{{\frac{\partial {#1}}{\partial {#2}}}}

\newcommand{\lder}[2]{{\partial {#1}/\partial {#2}}}
\newcommand{\dder}[3]{{\frac{\partial^2 {#1}}{\partial {#2}\partial {#3}}}}

\newcommand{\Z}{{\mathbb Z_{2}}}
\newcommand{\ZZ}{{\mathbb Z}}

\newcommand{\p}{\partial}

\newcommand{\w}{{\mathbf{w}}}

\newcommand{\widebar}{\overline}

\renewcommand{\a}{\alpha}

\newcommand{\e}{\varepsilon}
\newcommand{\s}{\sigma}
\newcommand{\f}{{\varphi}}
\renewcommand{\O}{\Omega}

\renewcommand{\o}{\omega}

\renewcommand{\t}{\theta}
\newcommand{\F}{{\Phi}}

\newcommand{\la}{{\lambda}}
\newcommand{\x}{{\xi}}

\renewcommand{\t}{\theta}

\newcommand{\itt}{{\tilde \imath}}

\newcommand{\at}{{\tilde a}}

\newcommand{\vt}{{\tilde v}}

\DeclareMathOperator{\Vect}{\mathrm{Vect}}

\newcommand{\tto}{{\linethickness{2pt}
		  \,\begin{picture}(1,0)
                   \put(0,0.26){\line(1,0){0.95}}
                   \put(0,0){$\boldsymbol{\rightarrow}$}
                  \end{picture}
                  }\,
}

\newcommand{\oto}{{\linethickness{0.5pt}
		  \,\begin{picture}(1,0)
		  \put(0.07,0.175){\line(0,1){0.2}}
                   \put(-0.01,0){$\boldsymbol{\Rightarrow}$}
                  \end{picture}
                  }\,
}

\newcommand{\ttto}[1]{\stackrel{#1}{\vphantom{\rightrightarrows}\tto}}

\newcommand{\otto}[1]{\stackrel{#1}{\vphantom{\rightrightarrows}\oto}}

\newcommand{\Sinf}{S_{\infty}}
\newcommand{\Pinf}{P_{\infty}}
\newcommand{\Linf}{L_{\infty}}

\begin{document}

\begin{abstract}
We show how the tangent functor extends   from ordinary smooth maps to ``microformal morphisms''   (also called ``thick morphisms'') of supermanifolds. Microformal morphisms generalize ordinary maps and correspond to formal canonical relations   between the cotangent bundles specified by generating functions   depending   on position variables on the source manifold and   momentum variables on the target manifold (as formal power expansions), regarded as   part of the structure. Microformal morphisms act on functions by non-linear (in general) pullbacks. 
We   obtain here non-linear pullbacks of (pseudo)differential forms and show that they respect the de Rham differentials as ``non-linear chain maps'' that can induce non-linear  transformations of cohomology.
\end{abstract}
%

\maketitle
\tableofcontents

\section{Introduction}
\subsection{}
In the series of papers~\cite{tv:nonlinearpullback,tv:oscil,tv:qumicro,tv:microformal, tv:gradedmicro} (see also~\cite{tv:thickspinor},\cite{tv:operovermap}) we introduced and studied a new type of morphisms between (super)manifolds,  generalizing  smooth maps, which we have called    \textbf{microformal} or \textbf{thick  morphisms}. Their distinctive feature is that they induce non-linear, in general, pullbacks on functions. More precisely, the  pullbacks are formal non-linear differential operators over ordinary maps. They are constructed by some iterative procedure. (For ordinary maps, which is a special case, the construction gives the familiar  linear pullbacks.) From algebraic viewpoint, these new type of pullbacks can be described as   \textbf{non-linear algebra homomorphisms}. By that we mean a   formal map  of algebras as vector spaces  such that for each point the derivative (which is a linear map) is an algebra homomorphism  in the usual sense. 

\subsection{}
Below we   recall   thick morphisms and   non-linear pullbacks induced by them.
For a  detailed introduction to the whole theory,   we refer the reader to the above-cited papers and particularly to~\cite{tv:microformal}.

A \emph{thick morphism} $\F\co M_1\tto M_2$ is defined in coordinates by a  \emph{generating function}   $S=S(x,q)$ depending on two groups of variables: local  coordinates $x^a$  on the source manifold and components of momentum  $q_i$ on the target mani\-fold. Generating functions are formal power series in the momentum variables, which we write as
\begin{equation}\label{eq.genfun}
  S(x,q)=S^0(x)+\f^i(x)q_i+\frac{1}{2}\,S^{ij}(x)q_jq_i+\ldots \,.
\end{equation}
The case when $S(x,q)$ is just a   linear function  in the momenta, $S(x,q)=\f^i(x)q_i$, corresponds to an ordinary map, $y^i=\f^i(x)$.  To a thick morphism $\F\co M_1\tto M_2$ we  assign a formal canonical relation  $T^*M_1\dashrightarrow T^*M_2$  between the cotangent bundles specified by $S$; with some abuse of language, one can identity a thick morphism with this relation. This is useful for   intuition, but because a generating function $S$ itself is taken as part of the structure, a thick morphism contains more information than   the corresponding relation,  which is determined   by the differential $dS$ only.

The justification of this notion is in the
construction of  \textbf{pullback}~\cite{tv:nonlinearpullback}. Heuristically, one has a canonical relation $T^*M_1\dashrightarrow T^*M_2$, and, if a function $g$ on $M_2$ is given, it corresponds to a canonical relation $T^*M_2\dashrightarrow \{*\}$; their composition will again be a canonical relation $T^*M_1\dashrightarrow \{*\}$, corresponding to a function on $M_1$ interpreted as  the desired pullback of $g$. By using the generating  function $S$, this can be made fully explicit as follows.

Given a function $g=g(y)$ on $M_2$, we define its \emph{pullback} $f=\F^*[g]$  as a function  on $M_1$, $f=f(x)$,    by
\begin{equation}\label{eq.main}
  f(x):=g(y)+S(x,q)-y^iq_i\,,
\end{equation}
where the extra variables $q_i,y^i$ are eliminated from the right-hand side by using the equations
\begin{equation}\label{eq.elim}
  q_i=\der{g}{y^i}(y)\,, \quad y^i=(-1)^{\itt} \der{S}{q_i}(x,q)\,.
\end{equation}
The latter system gives an equation for determining $y$ as a function of $x$ (and depending   on a function $g$)
\begin{equation}
  y^i=(-1)^{\itt} \der{S}{q_i}\Bigl(x,\der{g}{y}(y)\Bigr)\,,
\end{equation}
of the  ``fixed point'' type. It has a unique solution as a power series in the derivatives of $g$,
\begin{equation}
  y^i =\f^i_g(x)=\f^i(x)+S^{ij}(x)\p_jg(\f(x)) +\ldots \,.
\end{equation}
It is a formal perturbation $\f_g\co M_1\to M_2$, $\f_g=\f_0+\f_1+\ldots\;$,  of the map $\f\co M_1\to M_2$ given by $y^i =\f^i(x)$ and corresponding to $g=0$,  $\f=\f_0$.
Having obtained $y=\f_g(x)$ as a function of $x$, we substitute  it   into the first equation in~\eqref{eq.elim}, which makes it possible to eliminate both $y$ and $q$ from~\eqref{eq.main}.
For the pullback  $\F^*[g]$,  we  obtain   as the result a formal power-series expression in $g$ of the form
\begin{equation}
  f(x)= \underbrace{S^0(x)}_{\text{$0$th order}} +\underbrace{g(\f(x))}_{\text{$1$st order}} + \underbrace{\frac{1}{2}\,S^{ij}(x)\,\p_jg(\f(x))\p_ig(\f(x))}_{\text{$2$nd order}} + \ldots \ .
\end{equation}
Each term of order $k$ in $g$ contains the  derivatives of $g$ of orders up to $k-1$ evaluated at $y=\f(x)$. Thus the pullback
\begin{equation}
  \F^*\co \funn(M_2)\to \funn(M_1)
\end{equation}
is a formal nonlinear differential operator over the map $\f\co M_1\to M_2$  defined by the linear term in the generating function~\eqref{eq.genfun}.

For a linear generating function   $S(x,q)=\f^i(x)q_i$, the pullback defined by the above procedure   coincides with be the ordinary pullback $\f^*$ by the map $\f$ and in particular is   linear.    (The corresponding canonical relation $T^*M_1\dashrightarrow T^*M_2$ is   the cotangent lift $T^*\f$.)

\subsection{}
Because the pullbacks induced on smooth functions by thick morphisms are non-linear in general, there are actually two parallel theories: of  \textbf{``even''}  and  \textbf{``odd''  thick morphisms}  acting on even and odd functions (``bosonic'' and ``fermionic'' fields in physical parlance), respectively. Above, we  have  described the \textbf{even} case ($g$ should be assumed to be even); in the \textbf{odd} case, one has to use the anticotangent bundles  $\Pi T^*M_1$ and   $\Pi T^*M_2$ (where   $\Pi$ is the parity reversion functor) and   odd generating functions $S(x,y^*)$, where $y^*_i$ are the antimomenta (fiber coordinates for $\Pi T^*M_2$). We will use the notation $\Psi\co M_1\oto M_2$ for odd thick morphisms.

\subsection{}
The problem that we solve in the present paper is   extending   the \textbf{tangent functor} $T$ and the \textbf{antitangent functor} $\Pi T$ from ordinary smooth maps $\f\co M_1\to M_2$   to thick   morphisms of both types (even and odd).

Namely, we show that if $\F\co M_1\tto M_2$ is an even thick morphism, it induces an even thick morphism   $T\F\co TM_1\tto TM_2$ and an odd thick morphism $\Pi T\F\co \Pi T M_1\oto \Pi T M_2$ (generalizing $T\f$ and $\Pi T\f$ for an ordinary map), and this is a functorial correspondence.

Likewise, if $\Psi\co M_1\oto M_2$ is an odd thick morphism, it induces an odd thick morphism   $T\Psi\co TM_1\oto TM_2$   and an even thick morphism $\Pi T\Psi\co \Pi T M_1\tto \Pi T M_2$. (Note the ``change of  parity'' from odd to even for   the case of $\Pi T\Psi$.)

The tangent and antitangent thick morphisms $T\F$, $T\Psi$, $\Pi T\F$ and $\Pi T\Psi$   are fiberwise in a suitable sense. (In particular, they are morphisms over an ordinary map that is the `core' of a given thick morphism rather than over the thick morphism itself, see Section~\ref{sec.tang}).

Since functions in  $\fun(\Pi TM)$ in the case of ordinary manifolds can be identified with  differential forms $\O(M)$, and for supermanifolds they are by definition the Bernstein--Leites  pseudodifferential forms, we in particular obtain non-linear pullbacks on such forms. We show that these non-linear maps are compatible with the de Rham differentials  and induce non-linear transformations on  de Rham cohomology.

\subsection{}
Important role in our construction is played by a natural diffeomorphism  involving the cotangent bundle $T^*E$ of a vector bundle $E$. It was  discovered by W.~M.~Tulczyjew~\cite{tulczyjew:1977} for   $E=TM$ and   by K.~C.~H.~Mackenzie and P.~Xu~\cite{mackenzie:bialg} in  the general case  (see also~\cite{mackenzie:diffeomorphisms})\,\footnote{Also considered   in an unpublished text by J.-P.~Dufour; I  thank J.~Grabowski for sharing it with me.}.  Their  generalization to supermanifolds was established in~\cite{tv:graded}. Also in~\cite{tv:graded},   we  introduced an analog  for  the anticotangent bundle $\Pi T^*E$.

The paper is organized as follows.

In  Section~\ref{sec.natdif}, we elaborate the above-mentioned diffeomorphisms in an explicit form best suiting our needs, making a particular stress on the relations between tangent and cotangent bundles with various parity reversed combinations (i.e., ``odd analogs'' of the Tulczyjew map, see subsection~\ref{subsec.oddtul}).  The latter, to our knowledge,  have not previously appeared in the literature. In central Section~\ref{sec.tang}, we turn to    application to thick morphisms and work out all cases of tangent and antitangent functors. As a by-product, we introduce a general concept of a vector bundle thick morphism. See also our work in preparation~\cite{tv:comor}. In Section~\ref{sec.forms}, we look at the consequences for forms and de Rham cohomology. On the way, we elaborate the notions of a thick $Q$-morphism  and a ``non-linear chain map''.

\subsection{\textbf{Notation and terminology:}} we mostly follow the usage in~\cite{tv:microformal}. The parity ($\Z$-grading) of an object  is denoted by the tilde over the corresponding character. For local coordinates, their parities are assigned to the associated  indices. We suppress the prefix `super-' unless it is essential to stress the difference with the non-super case. So e.g. we speak about `manifolds' meaning `supermanifolds'. Also, although we do not stress this specifically, but  an extra $\ZZ$-grading or \emph{weight}, which is independent of parity, can be introduced everywhere;  so   our constructions hold for \emph{graded (super)manifolds}.

\section{Natural diffeomorphisms for  cotangent and anticotangent  bundles}
\label{sec.natdif}

\subsection{Mackenzie--Xu isomorphism and its odd analog}

The following theorem is due to Kirill Mackenzie and Ping Xu~\cite{mackenzie:bialg} for a general vector bundle $E$ (the special case  of $E=TM$ was obtained earlier by W.~Tulczyjew~\cite{tulczyjew:1977}). See also~\cite[Thm. 9.5.1]{mackenzie:book2005}.  The original versions were for ordinary manifolds, but the statement holds true for the super and graded cases~\cite{tv:graded}. We shall give a proof with explicit formulas.

To fix the notations: the canonical even symplectic form on the cotangent bundle $T^*M$ of a supermanifold $M$ is denoted $\o_{T*M}$; it is the differential of the canonical  odd $1$-form on $T^*M$ called the \emph{Liouville form}, which is denoted  $\t_M$,  so $\o_{T*M}=d\t_M$. If $x^a$ are local coordinates on $M$ and $p_a$ are the corresponding momenta, then $\t_M=p_adx^a$ and $\o_{T*M}=dp_adx^a$.   The invariance of the form $\t_M$ encodes the transformation law  for the momentum variables. (Note   that $p_adx^a=dx^a p_a$, because of the opposite parities of $p_a$ and $dx^a$.)

\begin{theorem} \label{thm.macxu}
For a vector bundle $E\to M$, there is a natural 
diffeomorphism $T^*(E)\cong T^*(E^*)$. It is an antisymplectomorphism of the canonical symplectic structures.
\end{theorem}
\begin{proof}
We shall introduce a transformation
\begin{equation}
  \kappa\co T^*E\to T^*E^*
\end{equation}
which will be the desired diffeomorphism.
Denote local coordinates on $M$ by $x^a$ and in the fibers of $E$ and $E^*$ by $u^i$ and $u_i$, respectively. We assume that the pairing between $E$ and $E^*$ is given by the invariant form $\langle u, u^*\rangle =u^iu_i$. (We view $u^i$ as left coordinates and $u_i$ as right coordinates in the dual frames. The distinction between `left' and `right' is important in the super case.) The corresponding conjugate momenta in $T^*E$ and $T^*(E^*)$ will be denoted by $p_a,p_i$ and $p_a,p^i$.  (Note   the same notation $p_a$  for coordinates on different manifolds.)   We have transformation laws for $u^i$ and $u_i$ as $u^i=u^{i'}T_{i'}{}^i$ and $u_i=T_{i}{}^{i'}u_{i'}$ with reciprocal matrices $T_{i'}{}^i$ and $T_{i}{}^{i'}$. Recall that the conjugate momenta transform as the partial derivatives in the respective coordinates. From here we have the transformation law for $p_i$, $p_i=T_{i}{}^{i'}p_{i'}$, which is the same as for $u_i$. 
Define $\kappa\co T^*E\to T^*E^*$ by the formulas
\begin{equation}\label{eq.kappa}
  \kappa^*(x^a)=x^a\,,  \  \kappa^*(u_i)=p_i\,, \ \kappa^*(p_a)=-p_a\,, \ \kappa^*(p^i)=(-1)^{\itt} u^i\,.
\end{equation}
To show that it this definition does not depend on a choice of coordinates, we  use the  Liouville $1$-forms.  On $T^*E$ we have 
\begin{equation}\label{eq.thetae}
    \t_{E}=dx^a p_a + du^ip_i\,,
\end{equation}
and on  $E^*$ we have, respectively,
\begin{equation}\label{eq.thetaestar}
    \t_{E^*}=dx^a p_a + du_ip^i\,.
\end{equation}
If we apply $\kappa^*$ to $\t_{E^*}$, we shall obtain
\begin{equation} \label{eq.leg}
    \kappa^*\t_{E^*}=-dx^a p_a + (-1)^{\itt}  dp_iu^i= -dx^a p_a + (-1)^{\itt}  d(p_iu^i)- p_idu^i= -\t_E+d(u^ip_i)\,.
\end{equation}
This proves that $\kappa$ is well-defined, and also that
\begin{equation}
  \kappa^*\o_{T^*E^*}= - \o_{T^*E}\,.
\end{equation}
Here $\o_{T^*E}= d\t_{E}=dp_adx^a+dp_idu^i$ and $\o_{T^*E^*}= d\t_{E^*}=dp_adx^a+dp^idu_i$ are the canonical symplectic forms. 
Hence
$\kappa\co T^*E\to T^*E^*$ is an antisymplectomorphism.
\end{proof}


We shall refer to the antisymplectomorphism $\kappa\co T^*E\to T^*E^*$ as the \emph{Mackenzie-Xu transformation}. Its action on  the Liouville forms, $\kappa^*\t_{E^*}= d(u^ip_i) -\t_E$, resembles    Legendre transform. 
When we need to emphasize the bundle $E$, we write $\kappa_E$ for $\kappa$.

With $E^{**}\cong E$, one may want to know about the transformation between the cotangent bundles going in the opposite direction. Caution is necessary in the super case because of signs. For clarity, let us express our formulas   writing fiber coordinates for all vector bundles as left coordinates; in particular, for $E^*$ we will have left coordinates as $\bar{u_i}=(-1)^{\itt} u_i$ (where $u_i$ are the right coordinates used in~\eqref{eq.kappa}), and the invariant pairing between $E$ and $E^*$ will take  the form $\langle u, u^*\rangle= u^i\bar{u_i}(-1)^{\itt}$.
Therefore, \emph{in terms of left coordinates only}, the Mackenzie-Xu transformation  $\kappa=\kappa_E\co T^*E\to T^*(E^*)$ is given by
  \begin{equation}\label{eq.kappaeleft}
  \kappa_E^*(\bar{u}_i)=(-1)^{\itt}p_i\,, \ \kappa_E^*(p_a)=-p_a\,, \ \kappa_E^*(\bar{p}^i)= u^i\,.
\end{equation}

For $E^{**}$ and $E^*$,    the pairing is $\langle u^*, u^{**}\rangle=  \bar{u_i}\bar{\bar{u}}^i(-1)^{\itt}=\bar{\bar{u}}^i\bar{u_i}$,  where  $\bar{\bar{u}}^i$ are the left coordinates in $E^{**}$. Hence the natural isomorphism of vector bundles $\s\co E^{**} \to E$ over $M$ is given by
\begin{equation}\label{eq.dualsq}
  \s^*(u^i):=(-1)^{\itt}\bar{\bar{u}}^i\,.
\end{equation}
We   denote by $\s$ also the induced diffeomorphism of the cotangent bundles,    $\s\co T^*(E^{**}) \to T^*E$, so   we have for it
\begin{equation}\label{eq.sigma}
   \s^*(u^i) =(-1)^{\itt}\bar{\bar{u}}^i\,, \  \s^*(p_i) =(-1)^{\itt}\bar{\bar{p}}_i\,
\end{equation}
(as well as $\s^*(x^a)=x^a$ and $\s^*(p_a)=p_a$).

Then the following holds.

\begin{proposition} \label{prop.kap}
Let $\kappa_E\co T^*E\to T^*(E^*)$ and $\kappa_{E^*}\co T^*(E^*)\to T^*(E^{**})$ be the Mackenzie-Xu transformations for the bundles $E$ and $E^*$, and let $\s\co T^*(E^{**}) \to T^*E$ be the natural diffeomorphism defined above.
Then the diffeomorphisms $\kappa_E$ and $\s\circ \kappa_{E^*}$,
\begin{equation}\label{eq.mxinv}
  T^*E \stackrel[\s\circ \kappa_{E^*}]{\kappa_E}{\rightleftarrows} T^*(E^*)\,,
\end{equation}
 are mutually inverse.
\end{proposition}
\begin{proof}
Similarly with~\eqref{eq.kappaeleft}, for $\kappa_{E^*}\co T^*(E^*)\to T^*(E^{**})$ we have
\begin{equation}\label{eq.kappaestarleft}
  \kappa_{E^*}^*(\bar{\bar{u}}^i)=(-1)^{\itt}\bar{p}^i\,, \ \kappa_{E^*}^*(p_a)=-p_a\,, \ \kappa_{E^*}^*(\bar{\bar{p}}_i)= \bar{u}_i\,.
\end{equation}
Combining this with the diffeomorphism $\s\co T^*(E^{**}) \to T^*E$ given by~\eqref{eq.sigma}, we
obtain
\begin{equation*}
  (\s\circ \kappa_{E^*})^*(u^i)=\bar{p}^i\,,  (\s\circ \kappa_{E^*})^*(p_a)=-p_a\,,   (\s\circ \kappa_{E^*})^*(p_i)=(-1)^{\itt}\bar{u}_i\,,
\end{equation*}
which is the inverse for~\eqref{eq.kappaeleft}.
\end{proof}

Theorem~\ref{thm.macxu} is best understood in the context of double vector bundles. As it is known~\cite{mackenzie:diffeomorphisms}, \cite[Ch. 9]{mackenzie:book2005}, for a vector bundle $E\to M$  the cotangent bundle $T^*(E)$ is   a double vector bundle over $M$ with the side bundles $E$ and $E^*$, and a similar double vector bundle structure is for $E^*$. Then the Mackenzie--Xu map $\kappa$ is an isomorphism of   double vector bundles:
\begin{equation} 
    {\begin{CD} T^*E@>>> E^*\\
                @VVV  @VVV \\
                E@>>>M
     \end{CD}\qquad}
    \xrightarrow{ \ \kappa \ }
    {\qquad \begin{CD}T^*(E^*)@>>> E^*\\
                @VVV  @VVV \\
                E @>>>M
    \end{CD}}\,.
\end{equation}
which is identical on $E\to M$, $E^*\to M$
and induces ``fiberwise $-1$'' on the core   $T^*M\to M$.  The second vector bundle structure in each of the double vector bundles actually encodes the Mackenzie-Xu isomorphism.   This also agrees with interpretation of double vector bundles as  (bi)graded manifolds~\cite{tv:qman-mack}.

\begin{remark} There is some flexibility with a choice of sign  in the definition of    $\kappa$. The choice   used here is the same as in the original definition of Mackenzie--Xu~\cite{mackenzie:bialg}, as well as~\cite{mackenzie:diffeomorphisms,mackenzie:book2005}, but is different from that  in~\cite{tv:graded} and \cite{tv:qman-mack}. Note that the equality~\eqref{eq.leg} corresponds to Theorem~9.5.2 in~\cite{mackenzie:book2005}. Also, there is  a ``quantum analog'' of the Mackenzie--Xu transformation~\cite{shemy:koszul} (for $\hbar$-differential operators).
\end{remark}

An ``odd version'' of the Mackenzie--Xu diffeomorphism
is as follows~\cite{tv:graded}. Recall that $\Pi$ is the parity reversion functor (on vector spaces, modules and vector bundles). We use the terminology such as  \emph{antitangent}  and \emph{anticotangent  bundles} for $\Pi TM$ and $\Pi T^*M$. Recall that the anticotangent bundle of any supermanifold $M$ is endowed with the canonical odd symplectic form, which we denote $\o_{\Pi T^*M}$. There is an even $1$-form on $\Pi T^*M$  analogous to   the Liouville $1$-form; we denote it  $\la_M$, so    $\o_{\Pi T^*M}=d\la_M$.  In local coordinates, $\la_M=dx^ax^*_a= (-1)^{\at+1}x^*_adx^a$, where $x^*_a$ are the  \emph{antimomenta}   canonically conjugate to coordinates  $x^a$ on $M$.   (The variables $x^*_a$ transform in the same way as $p_a$, i.e., as the partial derivatives in the respective coordinates, but have the opposite parity.)

\begin{theorem}[\cite{tv:graded}] \label{thm.oddmacxu}
For a vector bundle $E$, there is a natural diffeomorphism   $\Pi T^*E\cong \Pi T^*(\Pi E^*)$\,. It can be chosen as an antisymplectomorphism of the   odd symplectic structures.
\end{theorem}
\begin{proof}
  Denote by    $\x_i$   fiber coordinates in $\Pi E^*$, so that there is the odd bilinear form   $\langle u, \x^*\rangle= u^i\x_i$ is invariant. (That means that $\x_i$ are right coordinates and have transformation law $\x_i=T_{i}{}^{i'}\x_{i'}$, the same  as $u_i$. The parity of $\x_i$ is $\itt +1$, the opposite to that of $u^i$ or $u_i$.)  Let $x^*_a$, $u^*_i$ denote the  antimomenta conjugate to $x^a,u^i$ on $\Pi T^*E$, and let similarly $x^*_a$, $\x^{*i}$ denote the  antimomenta conjugate to $x^a,\x_i$ on $\Pi T^*(\Pi E^*)$.   Now we can introduce the desired diffeomorphism
  \begin{equation}\label{eq.oddkap}
    \kappa\co \Pi T^*E\to \Pi T^*(\Pi E^*)
  \end{equation}
 (for which we are using the same letter as for the even prototype), by 
 \begin{equation}\label{eq.oddkappa}
  \kappa^*(x^a)=x^a\,,  \  \kappa^*(\x_i)=u^*_i\,, \ \kappa^*(x^*_a)=-x^*_a\,, \ \kappa^*(\x^{*i})=  u^i
\end{equation}
 (note no signs depending on parities, unlike the even case~\eqref{eq.kappa}). To see that it is well-defined and possesses the desired properties,
 we   use the analogs   of the Liouville  forms. For $E$, we have
 \begin{equation}\label{eq.antilioue}
    \la_E= dx^a x^*_a + du^i u^*_i\,.
\end{equation}
 Similarly, for $\Pi E^*$ we have
\begin{equation}\label{eq.antiliouestar}
    \la_{\Pi E^*}= dx^a x^*_a + d\x_i \x^{*i}\,.
\end{equation}
 Then in the same way as above, we observe that from~\eqref{eq.oddkappa},
 \begin{equation}
   \kappa^*\la_{\Pi E^*}=-\la_{E} +d(u^iu^*_i)\,,
 \end{equation}
and everything follows from here.
\end{proof}

In the same way as above, the diffeomorphism of  $\Pi T^*E$ and $\Pi T^*(\Pi E^*)$ is best understood as an isomorphism of the double vector bundles
\begin{equation}  
    {\begin{CD} \Pi T^*E@>>> \Pi E^*\\
                @VVV  @VVV \\
                E@>>>M
    \end{CD}\qquad}
    \xrightarrow{ \ \kappa \ }
    {\qquad \begin{CD} \Pi T^*(\Pi E^*)@>>> \Pi E^*\\
                @VVV  @VVV \\
               E @>>>M
    \end{CD}}\,.
\end{equation}
There is also an analog of Proposition~\ref{prop.kap}.

\subsection{Tulczyjew isomorphism}

If one specializes the identification $T^*E\cong T^*(E^*)$ to the case   $E=TM$, this will give  $T^*(TM)\cong T^*(T^*M)$.
The canonical symplectic form on $T^*M$ makes it possible to raise and lower tensor induces and in particular to identify the cotangent bundle $T^*(T^*M)$ and the tangent bundle $T(T^*M)$. Combining that with the above isomorphism, we arrive  at a natural diffeomorphism $T^*(TM)\cong T(T^*M)$. Following Mackenzie~\cite{mackenzie:diffeomorphisms,mackenzie:book2005}, we call it the \emph{Tulczyjew isomorphism}.   As with the Mackenzie--Xu transformation and its odd analog, the Tulczyjew isomorphism is best understood in terms of double vector bundles. One can give a direct construction for it (close to Tulczyjew's original treatment). It is below.

\begin{theorem} \label{thm.macxutulcz}
For any supermanifold $M$, there is a
natural diffeomorphism
\begin{equation}\label{eq.tauabs}
  \tau\co T(T^*M)\to T^*(TM)\,,
\end{equation}
which is an  isomorphism of  double vector bundles:
\begin{equation}
  {\begin{CD} T(T^*M)@>>> T^*M\\
                @VVV  @VVV \\
                TM @>>>M
    \end{CD}\qquad}
    \xrightarrow{ \ \tau \ }
    {\qquad\begin{CD} T^*(TM)@>>> T^*M\\
                @VVV  @VVV \\
                TM @>>>M
    \end{CD}}
\end{equation}
\end{theorem}
\begin{proof}
  In coordinates, the double vector bundle structures in question are as follows:
  \begin{equation}\label{eq.ttstar}
  {\begin{CD} (x^a,p_a;\; \dot x^a,\dot p_a)@>>> (x^a,p_a)\\
                @VVV  @VVV \\
                (x^a,\dot x^a) @>>>(x^a)
    \end{CD}\quad}
    {\qquad\qquad}
    {\quad\begin{CD} (x^a,\dot x^a;\; p_a^{(1)}, p_a^{(2)})@>>> (x^a,p_a^{(2)})\\
                @VVV  @VVV \\
                (x^a,\dot x^a) @>>>(x^a)\,.
    \end{CD}}
\end{equation}
Each arrow is simply dropping part of the variables.
Here at the right-hand side, on $T^*(TM)$ the momentum variables $p_a^{(1)}$ are canonically conjugate with $x^a$, while $p_a^{(2)}$ are canonically conjugate with $\dot x^a$. Note that the pullback by the horizontal projection sends the   variables $p_a$ on $T^*M$ to the   variables $p_a^{(2)}$ on $T^*(TM)$ (not to $p_a^{(1)}$).  We define the map $\tau\co T(T^*M)\to T^*(TM)$   by
\begin{equation}\label{eq.tau}
  \tau^*(x^a)=x^a\,, \quad \tau^*(\dot x^a)=\dot x^a\,, \quad\tau^*(p_a^{(1)})= \dot p_a\,, \quad\tau^*(p_a^{(2)})=   p_a\,.
\end{equation}
On $T^*(TM)$, we have the Liouville $1$-form
\begin{equation}
  \t_{TM}=dx^a p_a^{(1)} + d\dot x^a p_a^{(2)}\,.
\end{equation}
On the other hand, on $T(T^*M)$  we have the   $1$-form $\dot \t_M$ obtained by formally differentiating the Liouville $1$-form $\t_M=dx^ap_a$ on $T^*M$,
\begin{equation}
  \dot\t_{M}=d\dot x^a p_a  + dx^a \dot p_a = dx^a \dot p_a + d\dot x^a p_a\,.
\end{equation}
By~\eqref{eq.tau}, we have   
\begin{equation}\label{eq.dotthetam}
  \tau^*(\t_{TM})= \dot\t_{M}\,.
\end{equation}
This   proves that~\eqref{eq.tau} is well-defined (and is a morphism of double vector bundles, i.e. commutes with the projections~\eqref{eq.ttstar}).
\end{proof}


We can use the   map $\tau$ for the identification of $T^*(TM)$ with $T(T^*M)$. In particular, we obtain a symplectic form $\o_{T(T^*M)}$ on $T(T^*M)$, defined as   $\tau^*(\o_{T^*(TM)})$\,,
\begin{equation}\label{eq.omttstar}
  \o_{T(T^*M)}:=
  \tau^*(\o_{T^*(TM)})= \tau^*(d\t_{TM}) =d (\tau^*(\t_{TM}))
   = d \dot\t_{M}= d \dot p_a dx^a   + dp_a d\dot x^a \,.
\end{equation}
We observe that   $d (\dot\t_{M})=  (d\t_{M})^{\dot{}}$,
so 
we  have
\begin{equation}\label{eq.dotomegatstarm}
  \o_{T(T^*M)}= \dot\o_{T^*M}\,.
\end{equation}
There is a special class of Lagrangian submanifolds in $T(T^*M)$ given by lifts to $T(T^*M)$ of Lagrangian submanifolds of $T^*M$\,:

\begin{corollary} \label{cor.tsubman}
Suppose $P\subset T^*M$ is a Lagrangian submanifold. Then $TP\subset T(T^*M)$ will be a Lagrangian submanifold of $T(T^*M)$ with respect to the symplectic structure~\eqref{eq.omttstar},\eqref{eq.dotomegatstarm}.
\end{corollary}
\begin{proof}  If $\o_{T^*M}$ vanishes on $P$, then on $TP$ the condition $\dot \o_{T^*M}=0$ will be satisfied. (As for   dimension, $\dim TP=2\dim P= \dim T^*M=\frac{1}{2}\dim T(T^*M)$.)
\end{proof}
For example, such will be the submanifolds $TM\subset T(T^*M)$ given by $p_a=0\,, \dot p_a=0$ and $T(T_{x_0}^*M)\subset T(T^*M)$ (for any $x_0\in M$) given by $x^a=x^a_0\,, \dot x^a=0 $.

\subsection{Three odd analogs of the Tulczyjew isomorphism} \label{subsec.oddtul}
Consider now   $E=\Pi TM$. We wish to give   descriptions for the cotangent and anticotangent bundles  $T^*(\Pi TM)$ and $\Pi T^*(\Pi TM)$ similar with the Tulczyjew isomorphism. Also, we shall supplement the above consideration of the Tulczyjew isomorphism by a description of $\Pi T^*(TM)$. We shall start from the latter. The     statements and proofs for all three cases will go very much in parallel.
As above, the main tool for us will be Liouville $1$-forms.  We shall use the same notation $\tau$ for all three odd  analogs of the Tulczyjew map that will be introduced.

Theorem~\ref{thm.oddmacxu} gives the  isomorphism  $\Pi T^*(TM)\cong  \Pi T^*(\Pi T^*M)$, and
raising indices  with the help of the odd symplectic form on $\Pi T^*M$  will give  $\Pi T^*(\Pi T^*M)\cong  T (\Pi T^*M)$, note the change of parity  because   the form is odd;     combined together, this gives $\Pi T^*(TM)\cong  T (\Pi T^*M)$. Similarly with the above, we can provide a direct construction.

\begin{theorem} \label{thm.oddtulcz}
For any supermanifold $M$, there is a
natural diffeomorphism
\begin{equation}
      \tau\co T (\Pi T^*M)\to \Pi T^*(TM)\,.
\end{equation}
which is an  isomorphism of  double vector bundles:
\begin{equation}\label{eq.tpitstar}
  {\begin{CD} T (\Pi T^*M) @>>> \Pi T^*M\\
                @VVV  @VVV \\
                  TM @>>>M
    \end{CD}\qquad}
    \xrightarrow{ \ \tau \ }
    {\qquad\begin{CD} \Pi T^*(TM)@>>> \Pi T^*M\\
                @VVV  @VVV \\
                TM @>>>M\,,
    \end{CD}}
\end{equation}
which in particular   endows $T(\Pi T^*M)$ with a canonical odd symplectic structure.
\end{theorem}
\begin{proof}
The double vector bundle structures in~\eqref{eq.tpitstar} are given by
 \begin{equation}
  {\begin{CD} (x^a,x^*_a;\; \dot x^a,\dot x^*_a)@>>> (x^a,x^*_a)\\
                @VVV  @VVV \\
                (x^a,\dot x^a) @>>>(x^a)
    \end{CD}\quad}
    {\qquad\qquad}
    {\quad\begin{CD} (x^a,\dot x^a;\; x^*_a, \dot x^*_a)@>>> (x^a,\dot x^*_a)\\
                @VVV  @VVV \\
                (x^a,\dot x^a) @>>>(x^a)\,,
    \end{CD}}
\end{equation}
similarly with~\eqref{eq.ttstar}.
At the left-hand side, the variables $\dot x^*_a$ are understood mnemonically as $(x^*_a)^{\cdot {}}$, while at the right-hand side,   $x^*_a$ are  conjugate to $x^a$ and $\dot x^*_a$ are  conjugate to $\dot x^a$. The horizontal arrow maps $(x^a,\dot x^a;\; x^*_a, \dot x^*_a)$ to $(x^a,x^*_a:=\dot x^*_a)$. Consider the  even analog of the  Liouville  form for $\Pi T^*M$,
$\la_{M}= dx^a x^*_a$\,.
It induces an invariant even $1$-form on $T(\Pi T^*M)$,
\begin{equation}
    \dot \la_{M}= dx^a \dot x^*_a +  d\dot x^a x^*_a\,.
\end{equation}
At the same time, there is the form $\la_{TM}$ on $\Pi T^*(TM)$,
\begin{equation}\label{laTM}
  \la_{TM}=dx^a x^*_a + d\dot x^a  \dot x^*_a\,.
\end{equation}
We define the map $\tau\co T (\Pi T^*M)\to \Pi T^*(TM)$   by
\begin{equation}\label{eq.oddtau}
  \tau^*(x^a)=x^a\,, \quad \tau^*(\dot x^a)=\dot x^a\,, \quad\tau^*(x^*_a)= \dot x^*_a\,, \quad\tau^*(\dot x^*_a)=   x^*_a\,,
\end{equation}
so that
\begin{equation}
  \tau^*(\la_{TM})=\dot \la_{M}\,,
\end{equation}
which in particular implies that the map $\tau$ specified by~\eqref{eq.oddtau} is well-defined. From here we obtain a canonical odd symplectic form $\o_{T(\Pi T^*M)}$ on $T(\Pi T^*M)$, as
\begin{equation} \label{eq.omegatpitstarm}
  \o_{T(\Pi T^*M)}:=d \dot \la_{M}= \tau^*(d\la_{TM})= \tau^*(\o_{\Pi T^*(TM)})\,,
\end{equation}
also
\begin{equation} \label{eq.dotomegapitstarm}
  \o_{T(\Pi T^*M)}=(\o_{\Pi T^*M})^{\cdot{}}\,,
\end{equation}
similarly with~\eqref{eq.dotomegatstarm}.
\end{proof}

\begin{corollary} \label{cor.oddtulcz}
For a Lagrangian submanifold $P\subset \Pi T^*M$, its tangent bundle $TP\subset T(\Pi T^*M)$ will be a Lagrangian submanifold of $T(\Pi T^*M)$ with respect to the odd symplectic structure~\eqref{eq.omegatpitstarm},\eqref{eq.dotomegapitstarm}.
\end{corollary}
\begin{proof} Indeed, if $\o_{\Pi T^*M}$ vanishes on $P$, then  $\dot \o_{\Pi T^*M}$ vanishes on $TP$.
\end{proof}

Before we proceed to the next two theorems, describing $T^*(\Pi TM)$ and $\Pi T^*(\Pi TM)$, we need to make some convention about notation.

Functions on the antitangent bundle $\Pi TM$ are (pseudo)differential forms on a (super)manifold $M$, for any $M$; we  have been normally using $dx^a$ for fiber coordinates  on $\Pi TM$ induced by local coordinates $x^a$ on $M$. Now we  will  have  to work with forms on $\Pi TM$ itself, i.e. with the iterated bundles such as $\Pi T(\Pi TM)$. A proper notation for these iterations (yielding `higher  differential forms' from the viewpoint of the original manifold) would be by using   operators such as $d_1$, $d_2$, etc., for each successive level. Since we do not need to go beyond the second iteration and for the sake of uniformity with the other cases, we shall keep $d$ for denoting \emph{forms} on  $\Pi TM$ (which otherwise would be denoted using $d_2$) and to change $dx^a$ to $\p x^a$ (which otherwise would be $d_1x^a$)   for fiber coordinates on $\Pi TM$\,, so that \emph{functions} on $\Pi TM$ will be written as functions of the variables $x^a,\p x^a$. It is worth noting that as odd operators, $d$ and $\p$ commute with the minus sign.

\begin{theorem}
\label{thm.antitulcz2}
For any supermanifold $M$, there is a
natural diffeomorphism
\begin{equation}
      \tau\co \Pi T (\Pi T^*M)\to T^*(\Pi TM)\,.
\end{equation}
which is an  isomorphism of  double vector bundles:
\begin{equation}\label{eq.pitpitstar}
  {\begin{CD}
                \Pi T (\Pi T^*M) @>>> \Pi T^*M\\
                @VVV  @VVV \\
                \Pi TM @>>>M
    \end{CD}\,\qquad}
    \xrightarrow{ \ \tau \ }
    {\qquad\begin{CD}
                T^*(\Pi TM)@>>> \Pi T^*M\\
                @VVV  @VVV \\
                \Pi TM @>>>M\,,
    \end{CD}}
\end{equation}
which in particular  endows $\Pi T (\Pi T^*M)$ with a canonical even symplectic form.
\end{theorem}
\begin{proof} Combining 
$T^*(\Pi TM)\cong T^*(\Pi T^*M)$ with $T^*(\Pi T^*M)\cong \Pi T(\Pi T^*M)$, 
we obtain $T^*(\Pi TM)\cong \Pi T(\Pi T^*M)$\,. To get a desired map $\tau$, 
we act as before. First of all, the double vector bundle structures are given in local coordinates by
\begin{equation}
  {\begin{CD} (x^a,x^*_a;\; \p x^a,\p x^*_a)@>>> (x^a,x^*_a)\\
                @VVV  @VVV \\
                (x^a,\p x^a) @>>>(x^a)
    \end{CD}\quad}
    {\qquad\qquad}
    {\quad\begin{CD} (x^a,\p x^a;\; p_a, \pi_a)@>>> (x^a,x^*_a:=\pi_a)\\
                @VVV  @VVV \\
                (x^a,\p x^a) @>>>(x^a)\,.
    \end{CD}}
\end{equation}
On $T^*(\Pi TM)$ there is the odd $1$-form (analog of the Liouville form)
\begin{equation}\label{eq.thetapitm}
  \t_{\Pi TM}= dx^a p_a + d(\p x^a)\pi_a\,,
\end{equation}
while on $\Pi T(\Pi T^*M)$ there is an odd $1$-form obtained as the  lift  of the
canonical even  $1$-form on $\Pi T^*M$, $\la_{M}=da^a x^*_a$\,:
\begin{equation}\label{eq.partiallambdaanticot}
    \p \la_{M}=-d(\p a^a)\, x^*_a +(-1)^{\at+1} dx^a\, \p x^*_a= 
    dx^a((-1)^{\at+1}\p x^*_a) + d(\p a^a)(-x^*_a)\,.
    \end{equation}
We define the desired diffeomorphism $\tau\co \Pi T (\Pi T^*M)\to T^*(\Pi TM)$ by
\begin{equation}\label{eq.oddtau2}
  \tau^*(x^a)=x^a\,, \quad \tau^*(\p x^a)=\p x^a\,, \quad\tau^*(p_a)= (-1)^{\at}\p x^*_a\,, \quad\tau^*(\pi_a)=   x^*_a\,.
\end{equation}
Then we obtain
\begin{equation}\label{eq.tautheta}
  \tau^*(\t_{\Pi TM}) = - \p \la_{M}= dx^a((-1)^{\at}\p x^*_a) + d(\p a^a)\,x^*_a\,.
\end{equation}
From here we deduce that $\tau$ is well-defined; we see that it commutes with the projections in the double vector bundle diagrams. Also, we can set
\begin{equation}\label{eq.omegapitpitstarm}
  \o_{\Pi T(\Pi T^*M)}:=d(- \p \la_{M})= \p (d\la_{M})=\p \o_{\Pi T^*M}\,,
\end{equation}
as an even symplectic form, and we have $\o_{\Pi T(\Pi T^*M)}=\tau^*(\o_{T^*(\Pi TM)})$.
%
%
\end{proof}

\begin{corollary} \label{cor.antitulcz}
For a Lagrangian submanifold $P\subset \Pi T^*M$ with respect to the odd symplectic form $\o_{\Pi T^*M}$, its antitangent bundle $\Pi TP\subset \Pi T(\Pi T^*M)$ will be a Lagrangian submanifold of $\Pi T(\Pi T^*M)$  with   respect to the even symplectic form $\o_{\Pi T(\Pi T^*M)}$ given by~\eqref{eq.omegapitpitstarm}.
\end{corollary}
\begin{proof} Indeed,  if $\o_{\Pi T^*M}$ vanishes on $P$, then  $\o_{\Pi T(\Pi T^*M)}=\p \o_{\Pi T^*M}$ vanishes on $\Pi TP$.
\end{proof}

\begin{theorem}
\label{thm.oddantitulcz}
For any supermanifold $M$, there is a natural diffeomorphism
\begin{equation}
      \tau\co \Pi T (T^*M)\to \Pi T^*(\Pi TM)\,.
\end{equation}
which is an  isomorphism of  double vector bundles:
\begin{equation}\label{eq.pittstar}
  {\begin{CD}
                \Pi T (T^*M) @>>>  T^*M\\
                @VVV  @VVV \\
                \Pi TM @>>>M
    \end{CD}\,\qquad}
    \xrightarrow{ \ \tau \ }
    {\qquad\begin{CD}
                \Pi T^*(\Pi TM)@>>> T^*M\\
                @VVV  @VVV \\
                \Pi TM @>>>M\,,
    \end{CD}}
\end{equation}
and in particular endows $\Pi T (T^*M)$ with a canonical odd symplectic form.
\end{theorem}

\begin{proof} Again we have  the canonical diffeomorphism $\Pi T^*(\Pi TM)\cong \Pi T^*(T^*M)$ and  raising indices with the help of the even symplectic form on $T^*M$ gives $\Pi T^*(T^*M)\cong \Pi T(T^*M)$, which together yield  $\Pi T^*(\Pi TM)\cong \Pi T(T^*M)$\,. An explicit identification can be obtained as follows. The double vector bundle structures are given in local coordinates by
\begin{equation}
  {\begin{CD} (x^a,p_a;\; \p x^a,\p p_a)@>>> (x^a,p_a)\\
                @VVV  @VVV \\
                (x^a,\p x^a) @>>>(x^a)
    \end{CD}\quad}
    {\qquad\qquad}
    {\quad\begin{CD} (x^a,\p x^a;\; x^*_a, \p x^*_a)@>>> (x^a,p_a:=\p x^*_a)\\
                @VVV  @VVV \\
                (x^a,\p x^a) @>>>(x^a)\,.
    \end{CD}}
\end{equation}
(here $\p x^*_a$ should be understood as $(\p x)^*_a$). Consider now on $\Pi T(T^*M)$ the even $1$-form
\begin{equation}\label{eq.partialthetam}
  \p \theta_M=\p (p_adx^a)= \p p_a dx^a +(-1)^{\at+1}p_ad(\p x^a)=
  dx^a((-1)^{\at}\p p_a) + d(\p x^a)(-p_a)\,,
\end{equation}
and on $\Pi T^*(\Pi TM)$, the canonical $1$-form
\begin{equation}\label{eq.lapitm}
  \la_{\Pi TM}=dx^a x^*_a+d(\p x^a)\,\p x^*_a\,.
\end{equation}
Define the desired diffeomorphism $\tau\co \Pi T(T^*M)\to \Pi T^*(\Pi TM)$ by
\begin{equation}\label{eq.oddtau3}
  \tau^*(x^a)=x^a\,, \quad \tau^*(\p x^a)=\p x^a\,, \quad\tau^*(x^*_a)= (-1)^{\at+1}\p p_a\,, \quad\tau^*(\p x^*_a)=   p_a\,.
\end{equation}
Then we obtain
\begin{equation}\label{eq.taula}
  \tau^*(\la_{\Pi TM}) =   dx^a((-1)^{\at+1}\p p_a)  + d(\p a^a)\,p_a=-\p(\t_M)\,.
\end{equation}
Hence  $\tau$ is well-defined and  it is a double vector bundle morphism. Also, if we   set
\begin{equation}\label{eq.omegapitstarm}
  \o_{\Pi T(T^*M)}:=d(- \p \t_{M})= \p (d\t_{M})=\p \o_{T^*M}\,,
\end{equation}
it will be  an odd symplectic form on $\Pi T(T^*M)$, and   $\tau^*(\o_{\Pi T^*(\Pi TM)})=\o_{\Pi T(T^*M)}$.
\end{proof}

\begin{corollary} \label{cor.oddantitulcz}
For a Lagrangian submanifold $P\subset   T^*M$ with respect to the even symplectic form $\o_{T^*M}$, its antitangent bundle $\Pi TP\subset \Pi T(T^*M)$ will be a Lagrangian submanifold  with the respect to the odd symplectic form~\eqref{eq.omegapitstarm} .
\end{corollary}
\begin{proof} Indeed,   if $\o_{T^*M}$ vanishes on $P$, then  $\o_{\Pi T(T^*M)}=\partial \o_{T^*M}$ vanishes on $\Pi TP$.
\end{proof}

\section{The tangent and antitangent morphisms for a thick morphism}
\label{sec.tang}

This is the main section; here we construct lifting of thick morphisms (even or odd) to tangent and antitangent bundles in a way generalizing the usual tangent map for   ordinary smooth maps. Crucial role will be played by Theorems~\ref{thm.macxutulcz},  \ref{thm.oddtulcz}, \ref{thm.antitulcz2} and \ref{thm.oddantitulcz} of the previous section.
There are some unexpected features making it different from ordinary maps: firstly, the tangent or antitangent thick morphisms obtained do  not `fiber' over the original thick morphism, rather over its core map (and there is a need for some technical assumption); secondly, the antitangent functor swaps even and odd thick morphisms (i.e., the antitangent to even will be odd, and vice versa). The whole notion of a `fiberwise thick morphism' of vector bundles with different bases has to be clarified from scratch, which we do in a special subsection~\ref{subsec.fiberthick}.


\subsection{The tangent functor $T$.}  
\subsubsection{Case of even thick morphisms.}
Our task is to define   lifting of a thick morphism between manifolds to the tangent bundles. There are two kinds of thick morphisms, `even' and `odd' (suitable for pulling back even and odd functions, respectively). We shall start from the even case. 

Recall that a thick morphism  $\Phi\co M_1\tto M_2$ is  described by an  even  generating function $S=S(x;q)$ depending on position variables $x^a$ on the source manifold and momentum variables $q_i$ on the target manifold. With respect to the momentum variables, it is a formal power series. To a thick morphism there corresponds   the formal  canonical relation specified by a generating function $S$, which  the Lagrangian submanifold of the product   $T^*M_2\times (-T^*M_1)$ given by the equation
\begin{equation}\label{eq.defphi}
    dy^iq_i- dx^a p_a = d(y^iq_i - S)\,.
\end{equation}
That means that
\begin{equation}
    y^i=(-1)^{\itt} \der{S}{q_i}(x;q)\,, \qquad p_a=\der{S}{x^a}(x;q)\,.
\end{equation}
We use the same letter $\Phi$ for this relation, but one should remember that the generating function contains more information than the Lagrangian submanifold (defined  by  $dS$, while the function $S$  includes a `constant of integration'). The relation $\Phi\subset T^*M_2\times (-T^*M_1)$ is formal because $S$ is a formal power series in $q_i$. One should think that we work in the formal neighborhood of the zero section in $T^*M_2$.

In order to extend the tangent functor to thick morphisms from ordinary smooth maps, we shall first analyze the tangent map for an ordinary map $\f\co M_1\to M_2$ in the language of relations and generating functions, and use this as a model. A map $\f\co M_1\to M_2$ corresponds to a canonical relation that we shall denote $R_{\f}$, $R_{\f}\subset T^*M_2\times (-T^*M_1)$, specified by the generating function $S=\f^i(x)q_i$  if $y^i=\f^i(x)$ is the coordinate expression of $\f$. With an abuse of language we shall refer to $S$ as  the generating function of $\f$. The tangent map $T\f\co TM_1\to TM_2$ is given by $y^i=\f^i(x), \dot y^i = \dot x^a \,\lder{\f^i}{x^a}(x)$. To write down the generating function for $T\f$, use Theorem~\ref{thm.macxutulcz} and the diffeomorphism $\tau$ for  the identification $T^*(TM)=T(T^*M)$. We can consider $T(T^*M_1)$ and $T(T^*M_2)$ as the cotangent bundles for $TM_1$ and $TM_2$, so that $\dot p_a, p_a$ and $\dot q_i,q_i$ are regarded as the conjugate momenta for $x^a,\dot x^a$ and $y^i,\dot y^i$, respectively.

\begin{proposition} For a smooth map $\f\co M_1\to M_2$, the canonical relation $R_{T\f}$ corresponding to the tangent map $T\f\co TM_1\to TM_2$ is  the tangent bundle $TR_{\f}$  for the canonical relation $R_{\f}$ corresponding  to $\f$. Here $TR_{\f}$ is regarded as a Lagrangian submanifold of $T(T^*M_2)\times (-T(T^*M_1))$\,. The generating function for $T\f$ is the symbolic time derivative of the generating function for $\f$,
\begin{equation}
    \dot S= \dot S(x,\dot x; \dot q, q)= {(\f^i(x)q_i)} \dot{\vphantom{S}}\,.
\end{equation}
\end{proposition}
\begin{proof} The generating function for $T\f$ is
\begin{equation*}
    S_{T\f}(x, \dot x; \dot q, q)= \f^i(x)\,\dot q_i + \dot x^a \,\der{\f^i}{x^a}(x)\,q_i= {(\f^i(x)q_i)} \dot{\vphantom{S}}\,,
\end{equation*}
as claimed. The Lagrangian  submanifold in $T(T^*M_2)\times (-T(T^*M_1))=T\bigl(T^*M_2\times (-T^*M_1)\bigr)$ it specifies is the tangent bundle $TR_{\f}$ of the submanifold $R_{\f}\subset T^*M_2\times (-T^*M_1)$\,.
\end{proof}

(From Corollary~\ref{cor.tsubman} we already know that the tangent bundle to a Lagrangian submanifold in $T^*M$ will be a Lagrangian submanifold in $T(T^*M)$.)

This motivates the following ``tautological'' definition.
\begin{definition} For a thick morphism $\Phi\co M_1\tto M_2$ specified by a generating function $S=S(x;q)$, the \emph{tangent morphism} is a thick morphism $T\Phi\co TM_1\tto TM_2$   specified as the generating function  by the formal `time  derivative' $\dot S$ of the  function $S$,
\begin{equation}\label{eq.dots}
    \dot S(x,\dot x; \dot q, q)= \dot x^a \der{S}{x^a}(x;q) + \dot q_i \der{S}{q_i}(x;q)\,.
\end{equation}
\end{definition}
Recall that with an abuse of notation we use the same symbol for a thick morphism and the corresponding relation (as a submanifold of the product of cotangent bundles). The following proposition ensures that this won't lead us into a problem for   tangent morphisms.
\begin{proposition}\label{prop.tphi}
The   canonical relation corresponding the tangent thick morphism $T\Phi$, as a submanifold in $T(T^*M_2)\times (-T(T^*M_1))$, is   the tangent bundle to the submanifold in $T^*M_2\times (-T^*M_1)$ corresponding to a thick morphism $\Phi$.
\end{proposition}
\begin{proof} Let us check directly that the generating function $\dot S$ given by~\eqref{eq.dots}
specifies the tangent bundle $T\Phi$, if $S$ specifies $\Phi\subset T^*M_2\times (-T^*M_1)$. We have, by the definition, equation~\eqref{eq.defphi}. By differentiating it formally with respect to $t$, we obtain
\begin{equation*}
    d\dot y^i q_i - d\dot x^a p_a +dy^i \dot q_i - dx^a \dot p_a = d(y^i q_i)\dot{}  +
    d\dot x^a\der{S}{x^a}-d\dot q_i\der{S}{q_i}- d x^a\!\left(\der{S}{x^a}\right)\dot{}-d q_i\!\left(\der{S}{q_i}\right)\dot{}\,,
\end{equation*}
which is equivalent to
\begin{equation*}
    -(-1)^{\itt}d\dot q_i y^i - d\dot x^a p_a - (-1)^{\itt}d q_i \dot y^i - dx^a \dot p_a =
    -d\dot x^a\der{S}{x^a}-d\dot q_i\der{S}{q_i}- d x^a\!\left(\der{S}{x^a}\right)\dot{}-d q_i\!\left(\der{S}{q_i}\right)\dot{}
\end{equation*}
or
\begin{equation*}
    d\dot x^a \left(\!p_a- \der{S}{x^a}\right) + d\dot q_i \left(\!(-1)^{\itt}y^i- \der{S}{q_i}\right)
    + dx^a \left(\!\dot p_a- \left(\der{S}{x^a}\right)\dot{}\right)
    +d q_i \left(\!(-1)^{\itt}\dot y^i -\left(\der{S}{q_i}\right)\dot{}\right)=0\,.
\end{equation*}
The first two terms   reproduce the equations specifying $\Phi$, while the other two terms give their tangent  prolongation. Altogether we have arrived at equations specifying $T\Phi$ as a submanifold in $T(T^*M_2)\times  T(T^*M_1)$, as expected. But at the same time, the result of the formal differentiation of equation~\eqref{eq.defphi} with respect to $t$  can be re-arranged differently, giving
\begin{equation*}
    d\dot y^i q_i +dy^i \dot q_i - d\dot x^a p_a  - dx^a \dot p_a =
    d\left(\dot y^iq_i + y^i\dot q_i -\dot S\right)\,.
\end{equation*}
Since $\dot q_i, q_i$ are the conjugate momenta for $y^i,\dot y^i$ and $\dot p_a, p_a$  are the conjugate momenta for   $x^a, \dot x^a$, we see it is the equation of the Lagrangian submanifold in $T(T^*M_2)\times  (-T(T^*M_1))$ with the generating function $\dot S(x, \dot x; \dot q, q)$. Hence the claim.
\end{proof}

We shall see that the construction of the tangent thick morphism  gives the expected properties such as   functoriality  (Theorem~\ref{thm.tasfunct} below); however,   the property of   $T\Phi$ being fiberwise  takes an unexpected form (see Theorem~\ref{thm.tphiasvbm} and Corollary~\ref{cor.basefunct} below). Also, in view of the non-linearity (in general) of the pullbacks by thick morphisms, it is not a priori obvious how a vector bundle structure can be respected. The answer is given by Theorem~\ref{thm.flin}.

\begin{theorem}
\label{thm.tphiasvbm}
Suppose in the expansion $S(x,q)=S^0(x)+\f^i(x)q_i+\ldots $ of the generating function of a thick morphism $\F\co M_1\tto M_2$ there is no zero-order term, i.e., $S(x,q)=\f^i(x)q_i+\ldots $, where the linear term defines a map $\f\co M_1\to M_2$. Then
there is a  commutative diagram
\begin{equation}\label{eq.tphiasvbm}
    \begin{CD} TM_1 & \ttto{T\Phi} & TM_2 \\
    @V{\pi_1}VV   @VV{\pi_2}V  \\
    M_1   @>>{\f}>   M_2\,.
    \end{CD}
\end{equation}
\end{theorem}
\begin{proof}
  Consider  $\F$ first without any extra assumptions, with the generating function $S(x,q)$ of the general form
  \begin{equation*}
    S(x,q)=S^0(x)+\f^i(x)q_i+\frac{1}{2}\,S^{ij}(x)q_jq_i+\ldots
  \end{equation*}
  Then $T\F$ is represented by the generating function $\dot S(x,\dot x; \dot q,q)$ given by~\eqref{eq.dots}. Consider diagram~\eqref{eq.tphiasvbm}. First   find the composition $\pi_2\circ T\F$. For that, we represent $\pi_2$ by its generating function  $S_{\pi_2}(y,\dot y; q)=y^iq_i$. Now, the generating function of the composition will be
  \begin{equation*}
    S_{\pi_2\circ T\F}(x,\dot x; q)=\dot S(x,\dot x; \widebar{\dot q},\widebar{q})+ S_{\pi_2}(\widebar{y},\widebar{\dot y}; q)-\widebar{y^i}\,\widebar{\dot q^i} -\widebar {\dot y^i}\,\widebar{q^i}\,,
  \end{equation*}
  where the variables $\widebar{y^i}$, $\widebar{\dot q_i}$, $\widebar {\dot y^i}$, $\widebar{q_i}$ are determined from the equations
  \begin{align*}
    \widebar{y^i} & =(-1)^{\itt}\der{\dot S}{\dot q_i}(x,\dot x; \widebar{\dot q},\widebar{q})\,, \\
    \widebar {\dot y^i} & = (-1)^{\itt}\der{\dot S}{q^i}(x,\dot x; \widebar{\dot q},\widebar{q})\,,\\
    \widebar {\dot q_i} & =  \der{S_{\pi_2}}{y^i}(\widebar{y},\widebar{\dot y};q)\,,\\
    \widebar {q_i} & =  \der{S_{\pi_2}}{\dot y^i}(\widebar{y},\widebar{\dot y};q)\,.
  \end{align*}
  Taking into account the explicit expressions for $\dot S$ and $S_{\pi_2}$ we obtain
  \begin{equation*}
    S_{\pi_2\circ T\F}(x,\dot x; q)=\dot x^a \der{S}{x^a}(x;\widebar{q}) + \widebar {\dot q_i} \der{S}{q_i}(x;\widebar{q})
    + \widebar{y^i}q_i - \widebar{y^i}\widebar{\dot q_i} - \widebar{\dot y^i}\widebar{q_i}\,,
  \end{equation*}
  where
  \begin{equation*}
    \widebar{y^i}   =(-1)^{\itt}\der{S}{q_i}(x,\widebar{q})\,, \quad
    \widebar {\dot q_i}   =  q_i\,,\quad
    \widebar {q_i}   =  0\,,
  \end{equation*}
  hence finally
  \begin{equation*}
    S_{\pi_2\circ T\F}(x,\dot x; q)= \dot x^a \der{S}{x^a}(x,0) +  q_i  \der{S}{q_i}(x,0)=
    \dot x^a \p_aS^0(x)+\f^i(x)q_i\,.
  \end{equation*}
  Under the assumption $S^0=0$, this will give  $S_{\pi_2\circ T\F}(x,\dot x; q)=\f^i(x)q_i$. This coincides with the generating function for the composition $\f\circ\pi_1\co TM_1\to M_2$.
\end{proof}

Hence for a thick morphism $\F\co M_1\tto M_2$, its tangent   $T\F\co TM_1\tto TM_2$ is   a morphism  over the  ordinary  map $\f\co M_1\to M_2$, not  over the thick morphism $\F$ itself, as one may wish to expect following too closely the case of usual maps\,\footnote{And as was wrongly claimed in the early version of this paper~\cite{tv:tangmicro}.}.

\smallskip
One needs the \textbf{assumption} $S^0\equiv 0$, \emph{and  unless otherwise stated, we assume that in the rest of the paper.}
\smallskip

\begin{corollary} \label{cor.basefunct}
Under the pullback by $T\F$, $(T\F)^*\co \funn(TM_2)\to \funn(TM_1)$, the functions on the base $\funn(M_2)\subset \funn(TM_2)$
are mapped to functions on the base $\funn(M_1)\subset \funn(TM_1)$,
and on them the restriction of the map $(T\F)^*$ is an algebra homomorphism.
\end{corollary}
\begin{proof}
Indeed, by the commutativity of \eqref{eq.tphiasvbm}, on $\funn(M_2)\subset \funn(TM_2)$, the pullback $(T\F)^*$ coincides with the usual pullback $\f^*$.
\end{proof}

Note that $(T\F)^*$ is in general non-linear; it will be instructive, however, to check its action on the fiberwise-linear functions on the tangent bundle $TM_2$. Together with the functions from $M_2$, they generate all other functions on $TM_2$ (more precisely, those that are fiberwise polynomial). To this end, we look closer at the expression~\eqref{eq.dots} for the generating function of $T\F$ and notice that it has a general form
\begin{multline}\label{eq.dotsdetail}
  \dot S(x,\dot x; \dot q, q)= \dot x^a\Bigl(\p_a\f^i(x)q_i+\frac{1}{2}\,\p_aS^{ij}(x)q_jq_i+ \ldots \Bigr) +
  \Bigl(\f^i(x)+S^{ij}(x)q_j+\frac{1}{2}\,S^{ijk}(x)q_kq_j+\ldots \Bigr)\dot q_i\\
  \equiv \dot x^a \underbrace{E_a(x,q)}_{\#q>0}+\underbrace{F^i(x,q)}_{\#q\geq 0}\dot q_i\,,
\end{multline}
where we have introduced a notation for the coefficients. Recall that $\dot q_i$ are conjugate to $y^i$ and $q_i$ are conjugate to $\dot y^i$ in $T(T^*M_2)$. Loosely, the term $F^i(x,q)\dot q_i$ corresponds to a map of the bases `with corrections'  given by the higher order terms in $q_i$ and the term $\dot x^a  E_a(x,q)$ corresponds to   thick morphisms of the fibers over each $x$ (which are ``fiberwise-linear'').

\begin{theorem} \label{thm.flin}
Under $(T\F)^*\co \funn(TM_2)\to \funn(TM_1)$, the fiberwise-linear functions on $TM_2$ are mapped to fiberwise-linear functions on $TM_1$ (in general, in a non-linear fashion).
\end{theorem}
\begin{proof}
To calculate $(T\F)^*$, we write $\dot S$ in the form~\eqref{eq.dotsdetail}, i.e. as $\dot S=\dot x^a  E_a(x,q)+ F^i(x,q)\dot q_i$. A function $g\in \funn(TM_2)$, $g=g(y,\dot y)$, is mapped to the function $f=f(x,\dot x)$ on $TM_1$ given by
\begin{equation*}
  f(x,\dot x)=g(y,\dot y)+ \dot S(x,\dot x\,; \dot q,q)-\dot y^i q_i- y^i\dot q_i\,,
\end{equation*}
where $\dot y^i$, $q_i$, $y^i$, and $\dot q_i$ are found from
\begin{align}
 y^i & = (-1)^{\itt}\der{\dot S}{\dot q_i} \equiv F^i(x,q)\,, \label{eq.yi}\\
  \dot y^i & = (-1)^{\itt}\der{\dot S}{q_i}\equiv (-1)^{\itt}\der{}{q_i}\bigl(\dot x^a  E_a(x,q)+ F^j(x,q)\dot q_j\bigr)\,, \label{eq.dotyi}\\
   q_i & = \der{g}{\dot y^i}(y,\dot y)\,, \label{eq.qi}\\
    \dot q_i & = \der{g}{y^i}(y,\dot y)\,. \label{eq.dotqi}
\end{align}
This gives
\begin{equation} \label{eq.fxdotx}
  f(x,\dot x)=g(y,\dot y)+  \dot x^a  E_a\bigl(x,\der{g}{\dot y}(y,\dot y)\bigr) - \dot y^i\der{g}{\dot y^i}(y,\dot y)\,,
\end{equation}
where $\dot y^i$, $y^i$ are found from~\eqref{eq.yi} and  \eqref{eq.dotyi} with $q$ and $\dot q$ substituted from~\eqref{eq.qi} and  \eqref{eq.dotqi}. Now, in the particular case of a fiberwise-linear function on $TM_2$, we have $g=\dot y^ig_i(y)$. Then $q_i=g_i(q)$,   the first and the last terms in~\eqref{eq.fxdotx} cancel, and we obtain
\begin{equation*}
  (T\F)^*[\dot y^ig_i(y)] = \dot x^a  E_a\bigl(x,g_i(y)\bigr)
\end{equation*}
where $y$ is determined from the equation
\begin{equation*}
  y^i= F^i\bigl(x,g_i(y)\bigr)
\end{equation*}
(here we are putting indices inside the arguments for the clarity of notation). We see that indeed the pullback of a fiberwise-linear function will be again fiberwise-linear, but with some nonlinear transformation of the coefficients.
\end{proof}


\begin{theorem}[functoriality] \label{thm.tasfunct}
For the composition of thick morphisms,
\begin{equation}
    T(\Phi_1\circ \Phi_2)= T\Phi_1 \circ T\Phi_2\,.
\end{equation}
\end{theorem}
\begin{proof}
Consider $\Phi_2\co M_1\tto M_2$ and $\Phi_1\co M_2\tto M_3$. Denote local coordinates on $M_3$ by $z^{\mu}$ and the corresponding momenta by $r_{\mu}$. Then the generating function of the composition is given by (see~\cite{tv:microformal})
\begin{equation}
  S_{\Phi_1\circ \Phi_2}(x,r)=S_{\Phi_2}(x,q)+S_{\Phi_1}(y,r) -y^iq_i\,,
\end{equation}
where $y^i$, $q_i$ are determined from the equations
\begin{equation}\label{eq.yq}
  y^i=(-1)^{\itt}\der{S_{\Phi_2}}{q_i}(x,q)\,, \quad q_i=\der{S_{\Phi_1}}{y^i}(y,r)\,.
\end{equation}
The generating function of the tangent thick morphism $T(\Phi_1\circ \Phi_2)$ will be the formal time derivative of $S_{\Phi_1\circ \Phi_2}$, hence given by the formal time derivative of the above, i.e.
\begin{multline}\label{eq.stcomp}
   S_{T(\Phi_1\circ \Phi_2)}= \dot S_{\Phi_1\circ \Phi_2}=(S_{\Phi_2}(x,q))^{\mathbf{\dot{}}}+(S_{\Phi_1}(y,r))^{\mathbf{\dot{}}} -\dot y^iq_i- y^i\dot q_i=\\
 S_{T\F_2}(x,\dot x, \dot q, q)+S_{T\F_1}(y,\dot y, \dot r, r)-\dot y^iq_i- y^i\dot q_i\,,
\end{multline}
where $y^i$, $q_i$ are determined from~\eqref{eq.yq} and $\dot y^i$, $\dot q_i$ are determined from
 \begin{align}
   \dot y^i & =(-1)^{\itt}\left(\der{S_{\Phi_2}}{q_i}(x,q)\right)^{\!\mathbf{\dot{}}}\,, \label{eq.yidotodin} \\
   \dot q_i & =\left(\der{S_{\Phi_1}}{y^i}(y,r)\right)^{\!\mathbf{\dot{}}}\,. \label{eq.qidotodin}
 \end{align}
On the other hand, the composition of the tangent morphisms $T\F_1$ and $T\F_2$ will be given by the generating function
\begin{equation}
  S_{T\Phi_1 \circ T\Phi_2}=S_{T\Phi_1}(y,\dot y,\dot r,r)+S_{T\Phi_2}(x,\dot x,\dot q, q)-\dot y^iq_i- y^i\dot q_i\,,
\end{equation}
which is the same as~\eqref{eq.stcomp}, and
where the variables $y^i$, $q_i$ are now determined from the system
\begin{align}
    y^i & =(-1)^{\itt}\der{}{\dot q_i}S_{T\F_2}\equiv (-1)^{\itt} \der{S_{\Phi_2}}{q_i}(x,q)\,, \label{eq.yidwa}\\
   \dot y^i & =(-1)^{\itt}\der{}{q_i}S_{T\F_2}\equiv (-1)^{\itt}\left(\dot x^a\dder{S_{\Phi_2}}{x^a}{q_i}(x,q)+
   \dot q_j\dder{S_{\Phi_2}}{q_j}{q_i}(x,q)\right)\,, \label{eq.yidot} \\
   q_i & =\der{}{\dot y^i}S_{T\F_1} \equiv \der{S_{\Phi_1}}{y^i}(y,r)\,,  \label{eq.qi1}\\
   \dot q_i & = \der{}{y^i}S_{T\F_1} \equiv  \dot y^j\dder{S_{\Phi_1}}{y^j}{y^i}(y,r)+
   \dot r_{\mu}\dder{S_{\Phi_1}}{r_{\mu}}{y^i}(y,r)\,. \label{eq.qidot}
 \end{align}
In this formulas,~\eqref{eq.yidwa} and~\eqref{eq.qi1} are the same as equations~\eqref{eq.yq}.  It remains to observe that~\eqref{eq.yidot} and~\eqref{eq.qidot} are equivalent to~\eqref{eq.yidotodin} and~\eqref{eq.qidotodin}, respectively (basically, ``time differentiation'' commutes with  partial derivatives).
\end{proof}

\subsubsection{Case of odd thick morphisms.}
The above construction of the tangent functor carries over to odd thick morphisms. Because of that, we only give statements and suppress proofs.

\begin{definition} For an odd thick morphism $\Psi\co M_1\oto M_2$ specified by an odd generating function $S=S_{\Psi}$,
\begin{equation}
    S=S(x;y^*)\,,
\end{equation}
the \emph{tangent morphism } $T\Psi$ is an odd thick morphism $T\Psi\co TM_1\oto TM_2$   specified as the generating function $S_{T\Psi}$ by the formal time  derivative  $\dot S$,
\begin{equation}
    S_{T\Psi}(x, \dot x; \dot y^*, y^*)= \dot S(x,\dot x; \dot y^*, y^*)= \dot x^a \der{S}{x^a}(x;y^*) + \dot y^*_i \der{S}{y^*_i}(x;y^*)\,.
\end{equation}
\end{definition}
Here we use    the canonical isomorphism $\Pi T^*(TM)\cong   T (\Pi T^*M)$ given by Theorem~\ref{thm.oddtulcz}  and also Corollary~\ref{cor.oddtulcz}. The corresponding Lagrangian submanifold is  the tangent bundle to the Lagrangian submanifold $\Psi \subset \Pi T^*M_2\times (-\Pi T^*M_1)$. (Compare with Proposition~\ref{prop.tphi}.) Note that on $T (\Pi T^*M_2)$, $\dot y^*_i$ is canonically conjugate (with respect to the odd symplectic form) to $y^i$, and  $y^*_i$ is canonically conjugate to $\dot y^i$.

Analogs of Theorems~\ref{thm.tphiasvbm} and \ref{thm.tasfunct} hold without surprises:

\begin{theorem}
\label{thm.tpsiasvbm}
Suppose, for an odd thick morphism $\Psi\co M_1\oto M_2$, in the expansion $S(x,y^*)=S^0(x)+\f^i(x)y^*_i+\ldots $ of the generating function  there is no zero-order term, i.e., $S(x,y^*)=\f^i(x)y^*_i+\ldots $, where the linear term defines a map $\f\co M_1\to M_2$.  Then 
there is a  commutative diagram
\begin{equation}\label{eq.tpsiasvbm}
    \begin{CD} TM_1 & \otto{T\Psi} & TM_2 \\
    @V{\pi_1}VV   @VV{\pi_2}V  \\
    M_1   @>>{\f}>   M_2\,.
    \end{CD}
\end{equation}
\end{theorem}

\begin{theorem}[functoriality] \label{thm.tasfunctodd}
For odd thick morphisms $\Psi_2\co M_1\oto M_2$ and $\Psi_1\co M_2\oto M_3$, the tangent to the composition is the composition of the tangents:
\begin{equation}
    T(\Psi_1\circ \Phi_2)= T\Psi_1 \circ T\Psi_2\,.
\end{equation}
\end{theorem}

Analogs of Corollary~\ref{cor.basefunct} and Theorem~\ref{thm.flin} can be packed into one statement:

\begin{theorem} \label{thm.flinodd}
Suppose an odd thick morphism $\Psi\co M_1\oto M_2$ is specified by an odd generating function of the form $S(x,y^*)=\f^i(x)y^*_i+\ldots $, i.e. without the zero-order term. Then under  $(T\Psi)^*\co \ofunn(TM_2)\to \ofunn(TM_1)$, the (odd) functions on the base, $\ofunn(M_2)\subset \ofunn(TM_2)$,
are mapped to  functions on the base,
$\ofunn(M_1)\subset \ofunn(TM_1)$,  by the algebra homomorphism $\f^*$, while   the odd fiberwise-linear functions on $TM_2$ are mapped to odd fiberwise-linear functions on $TM_1$ (in general, non-linearly).
\end{theorem}

\subsection{The antitangent functor $\Pi T$.} Although the idea is the same, here it is where we have some surprises. Namely, the functor $\Pi T$   swaps   even and odd thick morphisms.

Recall that, by Theorems~\ref{thm.antitulcz2} and \ref{thm.oddantitulcz}, there are
natural isomorphisms $ T^*(\Pi TM)\cong \Pi T (\Pi T^*M)$ and $\Pi T^*(\Pi TM)\cong \Pi T (T^*M)$, and by Corollaries~\ref{cor.antitulcz} and~\ref{cor.oddantitulcz}, the antitangent bundle to a Lagrangian submanifold in $T^*M_2\times (-T^*M_1)$ or $\Pi T^*M_2\times (-\Pi T^*M_1)$ will be a Lagrangian submanifold in
\begin{equation*}
  \Pi T(T^*M_2)\times (-\Pi T(T^*M_1))\cong \Pi T^*(\Pi TM_2)\times (-\Pi T^*(\Pi TM_1))
\end{equation*}
and
\begin{equation*}
  \Pi T(\Pi T^*M_2)\times (-\Pi T(\Pi T^*M_1))\cong   T^*(\Pi TM_2)\times (-T^*(\Pi TM_1))\,,
\end{equation*}
respectively. Hence,
for an even thick morphism
 \begin{equation*}
    \Phi\co M_1\tto M_2\,,
 \end{equation*}
the antitangent morphism $\Pi T\Phi$ will be an odd thick morphism
  \begin{equation*}
    \Pi T\Phi\co \Pi TM_1\oto \Pi TM_2\,,
 \end{equation*}
and
for an odd thick morphism
 \begin{equation*}
    \Psi\co M_1\oto M_2\,,
 \end{equation*}
the antitangent morphism $\Pi T\Psi$ will be an  even thick morphism
  \begin{equation*}
    \Pi T\Psi\co \Pi TM_1\tto \Pi TM_2\,.
 \end{equation*}
The respective generating functions are obtained from the generating functions of $\Phi$ and $\Psi$ by the application of the odd operator $\p$.

More formally, we have the following definitions:
\begin{definition}
For an even thick morphism $\Phi\co M_1\tto M_2$ specified by an  even  generating function $S=S(x;q)$, the \emph{antitangent morphism} 
$\Pi T\Phi\co \Pi TM_1\oto \Pi TM_2$  is an odd thick morphism specified  by the odd   generating function $S_{\Pi T\Phi}:=\p S$,
\begin{equation}\label{eq.sforpitphi}
    S_{\Pi T\Phi}(x, \p x; \p q, q)= \p S(x, \p x; \p q, q)= \p x^a \der{S}{x^a}(x;q) + \p q_i \der{S}{q_i}(x;q)\,.
\end{equation}
\end{definition}

\begin{definition}
For an odd thick morphism $\Psi\co M_1\oto M_2$ specified by an  odd  generating function $S=S(x;y^*)$, the \emph{antitangent morphism} 
$\Pi T\Psi\co \Pi TM_1\tto \Pi TM_2$  is an even thick morphism specified  by the even   generating function $S_{\Pi T\Psi}:=\p S$,
\begin{equation}\label{eq.sforpitpsi}
    S_{\Pi T\Psi}(x, \p x; \p y^*, y^*)= \p S(x, \p x; \p y^*, y^*)= \p x^a \der{S}{x^a}(x;y^*) + \p q_i \der{S}{y^*_i}(x;y^*)\,.
\end{equation}
\end{definition}

Recall from subsection~\ref{subsec.oddtul} that  
on $\Pi T(T^*M_2)$,  $\p q_i$ are conjugate to $y^i$ and $q_i$ are conjugate to $\p y^i$.
Similarly,
on $\Pi T(\Pi T^*M_2)$,  $\p y^*_i$ are conjugate to $y^i$ and $y^*_i$ are conjugate to $\p y^i$.

Analogs of Theorems~\ref{thm.tphiasvbm} and~\ref{thm.tpsiasvbm} hold (under the same assumptions on thick morphisms):

\begin{theorem}
\label{thm.pitphiasvbm}
Suppose, for an even thick morphism $\Phi\co M_1\tto M_2$, in the expansion $S(x,q)=S^0(x)+\f^i(x)q_i+\ldots $ of the generating function  there is no zero-order term, i.e., $S(x,q)=\f^i(x)q_i+\ldots $, where the linear term defines a map $\f\co M_1\to M_2$.  Then 
the diagram
\begin{equation}\label{eq.pitphiasvbm}
    \begin{CD} \Pi TM_1 & \otto{T\Phi} & \Pi TM_2 \\
    @V{\pi_1}VV   @VV{\pi_2}V  \\
    M_1   @>>{\f}>   M_2\,.
    \end{CD}
\end{equation}
is    commutative.
\end{theorem}

\begin{theorem}
\label{thm.pitpsiasvbm}
Suppose, for an odd thick morphism $\Psi\co M_1\oto M_2$, in the expansion $S(x,y^*)=S^0(x)+\f^i(x)y^*_i+\ldots $ of the generating function  there is no zero-order term, i.e., $S(x,y^*)=\f^i(x)y^*_i+\ldots $, where the linear term defines a map $\f\co M_1\to M_2$.  Then 
the diagram
\begin{equation}\label{eq.pitpsiasvbm}
    \begin{CD} \Pi TM_1 & \ttto{T\Psi} & \Pi TM_2 \\
    @V{\pi_1}VV   @VV{\pi_2}V  \\
    M_1   @>>{\f}>   M_2\,.
    \end{CD}
\end{equation}
is    commutative. 
\end{theorem}
Both theorems are proved similarly to   Theorem~\ref{thm.tphiasvbm} above.

Note that exactly because in the commutative diagrams~\eqref{eq.pitphiasvbm} and~\eqref{eq.pitpsiasvbm} (as well as in~\eqref{eq.tphiasvbm} and~\eqref{eq.tpsiasvbm}), the bottom arrow is the ordinary map $\f$ (the core of given thick morphisms) and \emph{not} the given even or odd thick morphism itself, the diagrams~\eqref{eq.pitphiasvbm} and~\eqref{eq.pitpsiasvbm} make  sense. Otherwise   swapping of even and odd thick morphisms by $\Pi T$ would be a problem.

Similarly to the tangent morphisms $T\Phi$ and $T\Psi$ considered above, the antitangent morphisms $\Pi T\Phi$ and $\Pi T\Psi$ share the property of being ``thick morphisms of vector bundles'' (see   discussion of a general concept in the next subsection). Namely, we have the following analogs of Corollary~\ref{cor.basefunct}, Theorem~\ref{thm.flin}, and Theorem~\ref{thm.flinodd}:

\begin{theorem} \label{thm.flinantiev}
Suppose an even thick morphism $\Phi\co M_1\oto M_2$ is specified by an even generating function  $S(x,y^*)=\f^i(x)y^*_i+\ldots $ without zero-order term. Then under the pullback
\begin{equation*}
  (\Pi T\Phi)^*\co \ofunn(\Pi TM_2)\to \ofunn(\Pi TM_1)\,,
\end{equation*}
the    functions from $\ofunn(M_2)$
are mapped to    functions from $\ofunn(M_1)$   by the algebra homomorphism $\f^*$, and   the  (odd) fiberwise-linear functions  on $\Pi TM_2$ are mapped to  (odd) fiberwise-linear functions  on $\Pi TM_1$ (in general, non-linearly).
\end{theorem}

\begin{theorem} \label{thm.flinantiod}
Suppose an odd thick morphism $\Psi\co M_1\oto M_2$ is specified by an odd generating function  $S(x,y^*)=\f^i(x)y^*_i+\ldots $ without zero-order term. Then under  the pullback
\begin{equation*}
  (\Pi T\Psi)^*\co \funn(\Pi TM_2)\to \funn(\Pi TM_1)
\end{equation*}
 the functions from $\funn(M_2)$
are mapped to    functions from $\funn(M_1)$  by the algebra homomorphism $\f^*$, and   the (even)  fiberwise-linear functions on  $\Pi TM_2$ are mapped to  (even) fiberwise-linear functions on  $\Pi TM_1$ (in general, non-linearly).
\end{theorem}
The proofs go along the same lines as for Corollary~\ref{cor.basefunct} and Theorem~\ref{thm.flin}. See also in  the next subsection.

Finally, we have the functoriality as expected:
\begin{theorem} \label{thm.pitasfunctev}
For even thick morphisms  $\Phi_2\co M_1\tto M_2$ and $\Phi_1\co M_2\tto M_3$,
\begin{equation}\label{eq.pitphi1phi2}
    \Pi T(\Phi_1\circ \Phi_2)=\Pi T\Phi_1 \circ \Pi T\Phi_2\,.
\end{equation}
\end{theorem}
Both sides of~\eqref{eq.pitphi1phi2} are odd thick morphisms $\Pi TM_1\oto \Pi TM_3$.

\begin{theorem} \label{thm.pitasfunctod}
For odd thick morphisms  $\Psi_2\co M_1\oto M_2$ and $\Psi_1\co M_2\oto M_3$,
\begin{equation}\label{eq.pitpsi1psi2}
    \Pi T(\Psi_1\circ \Psi_2)=\Pi T\Psi_1 \circ \Pi T\Psi_2\,.
\end{equation}
\end{theorem}
Both sides of~\eqref{eq.pitpsi1psi2} are even thick morphisms $\Pi TM_1\tto \Pi TM_3$.

Proofs of Theorem~\ref{thm.pitasfunctev} and Theorem~\ref{thm.pitasfunctod} are similar to that of Theorem~\ref{thm.tasfunct}.

\subsection{Thick morphisms of vector bundles}\label{subsec.fiberthick}

In this subsection, we somewhat digress from our main subject and introduce a general notion of a ``vector bundle thick morphism'' modelling it on the particular examples we have encountered in this paper. More detailed analysis is given in our joint paper in preparation~\cite{tv:comor}.

In the classical setup, the tangent map (for a given smooth map) is a morphism of vector bundles over the given map of bases. As we have seen above, the analog of that in microformal geometry is not   straightforward. If we want to generalize, we observe, taking into account that thick morphisms are not maps of sets, that it is not so obvious  what    \emph{fiberwise}   should mean for thick morphisms of fiber bundles. Even in the example for the tangent morphism, where the bases are related by an ordinary map, 
we do not arrive at what one would  imagine, i.e. a family of thick morphisms of the fibers of  the form $E_{1x}\tto E_{2\f(x)}$. So much for a ``fiberwise'' condition;  the next question is how to approach   fiberwise \emph{linearity}  in the case of vector bundles. The same example of tangent morphisms shows that it is possible to have linearity    respected in a certain particular sense, in spite of the corresponding pullback's being non-linear.

Below we put forward a definition that may look   arbitrary, but, as we show, possesses   desired properties (maybe in a modified form). We consider the case of even thick morphisms. The case of odd thick morphisms is similar.

Let $E_1$ and $E_2$ be vector bundles over   bases $M_1$ and $M_2$. (As usual, everything can be in the category of supermanifolds.) Denote local coordinates on $M_1$ and $M_2$ by $x^a$ and $y^i$, respectively. Denote linear coordinates in the fibers of $E_1$ and $E_2$, by $u^{\a}$ and $w^{\mu}$. The corresponding conjugate momenta on $T^*E_1$ and $T^*E_2$ will be denoted by $p_a$, $p_{\a}$, $q_{i}$ and $q_{\mu}$.

\begin{definition} \label{def.fiblingen}
A thick morphism $\F\co E_1\tto E_2$ is \emph{fiberwise-linear} if it is specified by a generating function of the form\footnote{For clarity though at the expense of beauty, we are putting here variables with indices as arguments.}
\begin{equation}\label{eq.fiblingen}
  S(x^a,u^{\a};q_i,q_{\mu})= S^i(x^a,q_{\mu})q_i + u^{\a}S_{\a}(x^a,q_{\mu})\,,
\end{equation}
where $S_{\a}(x^a,0)=0$.
\end{definition}

The coefficients $S^i(x^a,q_{\mu})$ and $S_{\a}(x^a,q_{\mu})$ are formal power series in $q_{\mu}$\,:
\begin{equation}\label{eq.si}
  S^i(x^a,q_{\mu})= S^i(x)+ S^{i\mu}(x)q_{\mu}+ \frac{1}{2}\,S^{i\mu\nu}(x)q_{\nu}q_{\mu}+\ldots
\end{equation}
and
\begin{equation}\label{eq.salpha}
  S_{\a}(x^a,q_{\mu})=S_{\a}^{\mu}(x)q_{\mu}+ \frac{1}{2}\,S_{\a}^{\mu\nu}(x)q_{\nu}q_{\mu} +\ldots
\end{equation}

\begin{example} \label{ex.usual}
Set
\begin{equation}\label{eq.fiblinmap}
  S(x^a,u^{\a};q_i,q_{\mu})= \f^i(x)q_i + u^{\a}\F_{\a}^{\mu}(x)q_{\mu}\,,
\end{equation}
so $S^i(x^a,q_{\mu})=S^i(x)\equiv\f^i(x)$ and $S_{\a}(x^a,q_{\mu})=S_{\a}^{\mu}(x)q_{\mu} \equiv \F_{\a}^{\mu}(x)q_{\mu}$.
This corresponds to an ordinary fiberwise map $\F\co E_1\to E_2$ linear on each fiber, over a map $\f\co M_1\to M_2$,    given by the formulas $y^i=\f^i(x)$ and $w^{\mu}=u^{\a}\F_{\a}^{\mu}(x)$.
\end{example}

A general fiberwise-linear thick morphism given by~\eqref{eq.fiblingen} can be seen a `thickening' of a usual fiberwise-linear map of vector bundles $E_1\to E_2$ over a map $M_1\to M_2$, as in Example~\ref{ex.usual}, where  $\f^i(x)\equiv S^i(x)$ of \eqref{eq.si} and $\F_{\a}^{\mu}(x)\equiv S_{\a}^{\mu}(x)$ of \eqref{eq.salpha}.

\begin{example} Various versions of the tangent morphisms considered above,   $T\F$, $T\Psi$, $\Pi T\F$, and  $\Pi T\Psi$, all have the generating functions of
the form~\eqref{eq.fiblingen} or its odd analog. In particular, for $T\F$ we had the generating function~\eqref{eq.dotsdetail}, i.e.
\begin{equation}
  \dot S(x,\dot x; \dot q, q)=  F^i(x,q) \dot q_i + \dot x^a  E_a(x,q)
\end{equation}
 where
\begin{equation*}
  F^i(x,q)=\f^i(x)+S^{ij}(x)q_j+\frac{1}{2}\,S^{ijk}(x)q_kq_j+\ldots
\end{equation*}
and
\begin{equation*}
  E_a(x,q)=\p_a\f^i(x)q_i+\frac{1}{2}\,\p_aS^{ij}(x)q_jq_i+ \ldots
\end{equation*}
Here   $\dot x^a$ play the role of $u^{\a}$ and $\dot y^i$ play the role of $w^{\mu}$, and recall that here $\dot q_i$ are conjugate to $y^i$, so play the role of $q_i$ in~\eqref{eq.fiblingen}, while $q_i$ are conjugate to $\dot y^i$, so play the role of $q_{\mu}$.
\end{example}

\begin{example}[special case of Definition~\ref{def.fiblingen}]
\label{def.fiblin}
Suppose a thick morphism $\F\co E_1\tto E_2$ has a generating function of the form
\begin{equation}\label{eq.fiblin}
  S(x^a,u^{\a};q_i,q_{\mu})= \f^i(x)q_i + u^{\a}S_{\a}(x^a,q_{\mu})\,.
\end{equation}
The difference with the general case~\eqref{eq.fiblingen} is that in~\eqref{eq.fiblin}, $S^i(x^a,q_{\mu})\equiv \f^i(x)$ do  not depend on the momentum variables $q_{\mu}$. One can see that here the thick morphism $\F$ reduces to a family of thick morphisms  $\F_x\co E_{1x}\tto E_{2\f(x)}$ of the fibers, where $\f\co M_1\to M_2$ is a map specified by $y^i=\f^i(x)$. (The precise  meaning is that such is the action of $\F$ on functions.)
\end{example}

Note that the tangent morphisms do  \emph{not} fall under Example~\ref{def.fiblin}.

First we need to know that the definition of fiberwise-linear morphisms does not depend on   choices of coordinates and that such morphisms are composable.

\begin{theorem}\label{thm.flm}
  The class of fiberwise-linear thick morphisms of  vector bundles is well-defined, i.e.  generating functions of the form~\eqref{eq.fiblingen}  are invariant under changes of coordinates in vector bundles.  It is closed under composition of thick morphisms.
\end{theorem}
(Transformation law for generating functions of thick morphisms is given in~\cite{tv:microformal}.)

We omit the proof of Theorem~\ref{thm.flm} (see~\cite{tv:comor}), but we will give  a proof of the following theorem.

\begin{theorem}\label{thm.fl}
  Let $\F\co E_1\tto E_2$ be a fiberwise-linear thick morphism of vector bundles with a generating function~\eqref{eq.fiblingen}. Then:

  (1)  the diagram
\begin{equation}\label{eq.diagflm}
  \begin{CD} E_1 & \ttto{\F} & E_2 \\
    @V{\pi_1}VV   @VV{\pi_2}V  \\
    M_1   @>>{\f_0}>   M_2\,.
    \end{CD}
\end{equation}
is commutative, where the $\f_0\co M_1\to M_2$ is given by $y^i=\f^i(x)\equiv S^i(x)$  from~\eqref{eq.si}; and

(2) for a  function of the form $g(y,w)=g_0(y)+ w^{\mu}g_{\mu}(y)$ on $E_2$ (``fiberwise-affine''), its pullback by $\F$ will be a function on $E_1$ of the same form:
  \begin{equation}\label{eq.phistaraffine}
    \F^*[g]=f_0(x) + u^{\a}f_{\a}(x)= g_0(y) + u^{\a}S_{\a}\bigl(x;g_{\mu}(y)\bigr)\,,
  \end{equation}
where in~\eqref{eq.phistaraffine}, $y$ is the function of $x$ found from the equations
\begin{equation}\label{eq.fg}
  y^i=S^i(x;g_{\mu}(y))\,.
\end{equation}
In particular, for functions on the base, $g=g(y)$, there will be $\F^[g]=\f_0^*(g)$; and for fiberwise-linear functions, $g=w^{\mu}g_{\mu}(y)$, there will be  $\F^[g]=u^{\a}S_{\a}\bigl(x;g_{\mu}(y)\bigr)$, where $y$ is found from~\eqref{eq.fg}.
\end{theorem}
\begin{proof}
To prove part (1) of the theorem, consider the generating functions for all the arrows in diagram~\eqref{eq.diagflm}. For the top, we have~\eqref{eq.fiblingen}. For the bottom, it is $S_{\f_0}(x;q_i)=\f^i(x)q_i$. For the vertical arrows, we have $S_{\pi_1}(x^a,u^{\a};\bar p_a)=x^a\bar p_a$ and $S_{\pi_2}(y^i,w^{\mu};\bar q_i)=y^i\bar q_i$. Therefore, we obtain the following for the compositions:
\begin{multline*}
  S_{\pi_2\circ \F}(x^a,u^{\a};q_i)= S(x^a,u^{\a};\bar q_i,\bar q_{\mu}) + S_{\pi_2}(\bar y^i,\bar w^{\mu};  q_i) -\bar y^i\bar q_i- \bar w^{\mu}\bar q_{\mu}= \\
  S^i(x^a,\bar q_{\mu})\bar q_i + u^{\a}S_{\a}(x^a,\bar q_{\mu}) + \bar y^i q_i
  -\bar y^i\bar q_i- \bar w^{\mu}\bar q_{\mu}\,,
\end{multline*}
where $\bar q_i$, $\bar q_{\mu}$, $\bar y^i$, and $\bar w^{\mu}$ are determined from the equations
\begin{align*}
    \bar q_i & = \der{}{\bar y^i}S_{\pi_2}(\bar y^i,\bar w^{\mu};  q_i)\equiv q_i\,,  \\
   \bar q_{\mu} & = \der{}{\bar w^{\mu}}S_{\pi_2}(\bar y^i,\bar w^{\mu};  q_i)\equiv 0\,,  \\
   \bar y^i & =(-1)^{\itt}\der{}{\bar q_i}S(x^a,u^{\a};\bar q_i,\bar q_{\mu}) \equiv S^i(x^a,\bar q_{\mu})\,, \\
   \bar w^{\mu} & = (-1)^{\tilde\mu} \der{}{\bar q_{\mu}}S(x^a,u^{\a};\bar q_i,\bar q_{\mu})\,,
 \end{align*}
hence we have
\begin{equation*}
  S_{\pi_2\circ \F}(x^a,u^{\a};q_i)= S^i(x^a,0)  q_i + u^{\a}S_{\a}(x^a,0)= S^i(x)q_i +0=\f^i(x)q_i\,,
\end{equation*}
from~\eqref{eq.si} and \eqref{eq.salpha}, and similarly
\begin{equation*}
  S_{\f_0\circ \pi_1}(x^a,u^{\a};q_i)=   S_{\pi_1}(x^a,u^{\a};  \bar p_a) + S_{\f_0}(\bar x^a; q_i)- \bar x^a\bar p_a= x^a\bar p_a+ \f^i(\bar x)q_i - \bar x^a\bar p_a\,,
\end{equation*}
where $\bar x^a$ and $\bar p_a$ are determined from
\begin{align*}
    \bar x^a & =(-1)^{\at}\der{}{\bar p_a}S_{\pi_1}(x^a,u^{\a};  \bar p_a)\equiv x^a\,,\\
  \bar p_a & = \der{}{\bar x^a}S_{\f_0}(\bar x^a; q_i)\,,
\end{align*}
hence
\begin{equation*}
  S_{\f_0\circ \pi_1}(x^a,u^{\a};q_i) = \f^i(x)q_i\equiv S_{\pi_2\circ \F}(x^a,u^{\a};q_i)\,,
\end{equation*}
which proves the commutativity of~\eqref{eq.diagflm}. (Note that we substantially relied on the condition $S_{\a}(x^a,0)=0$, which is part of Definition~\ref{def.fiblingen}; otherwise   in the final formula for $S_{\pi_2\circ \F}(x^a,u^{\a};q_i)$ there would be an extra term  obstructing the commutativity.)

To prove part (2), consider an even function on $E_2$ of the form $g(y^i,w^{\mu})=g_0(y)+w^{\mu}g_{\mu}(y)$. Its fullback by $\F$ will be the function $f(x^a,u^{\a})$ on $E_1$ given by
\begin{multline*}
  f(x^a,u^{\a})=  S(x^a,u^{\a};\bar q_i,\bar q_{\mu}) + g(\bar y^i,\bar w^{\mu}) -\bar y^i\bar q_i- \bar w^{\mu}\bar q_{\mu}=\\
  S^i(x^a,\bar q_{\mu})\bar q_i + u^{\a}S_{\a}(x^a,\bar q_{\mu}) + g_0(\bar y)+\bar w^{\mu}g_{\mu}(\bar y)
  -\bar y^i\bar q_i- \bar w^{\mu}\bar q_{\mu}\,,
\end{multline*}
where $\bar q_i$, $\bar q_{\mu}$, $\bar y^i$, and $\bar w^{\mu}$ are determined from
\begin{align*}
    \bar q_i & = \der{}{\bar y^i}g(\bar y^i,\bar w^{\mu})\,,  \\
   \bar q_{\mu} & = \der{}{\bar w^{\mu}}g(\bar y^i,\bar w^{\mu})\equiv g_{\mu}(\bar y)\,,  \\
   \bar y^i & =(-1)^{\itt}\der{}{\bar q_i}S(x^a,u^{\a};\bar q_i,\bar q_{\mu}) \equiv S^i(x^a,\bar q_{\mu})\,, \\
   \bar w^{\mu} & = (-1)^{\tilde\mu} \der{}{\bar q_{\mu}}S(x^a,u^{\a};\bar q_i,\bar q_{\mu})\,,
 \end{align*}
hence we obtain that, for our $g$,
\begin{equation*}
  \F^*[g]=f(x,u)= u^{\a}S_{\a}(x^a,  q_{\mu}) + g_0(y)
\end{equation*}
(we were able to get rid of bars over the variables, now redundant)
where $q_{\mu}=g_{\mu}(y)$ and $y$ is determined from the equations $y^i=S^i(x; g_{\mu}(y))$. This is what is claimed in~\eqref{eq.phistaraffine} and \eqref{eq.fg}. In particular, when $g_{\mu}(y)\equiv 0$, i.e. we have a function on the base $g=g_0(y)$, then $y$ is obtained explicitly as $y=S^i(x,0)=\f^i(x)$, hence $\F^*[g]=\F^*[g_0]$ is the ordinary pullback $\f_0^*(g_0)$. (This also follows from part (1).) On the other hand, when $g_0(x)\equiv 0$, so we have a fiberwise-linear function on $E_2$, $g=w^{\mu}g_{\mu}(y)$, we arrive at the function
\begin{equation*}
  \F[w^{\mu}g_{\mu}(y)]=u^{\a}f_{\a}(x)\equiv u^{\a}S_{\a}(x;g_{\mu}(y))\,,
\end{equation*}
where $y$ is given as the solution of $y^i=S^i(x;g_{\mu}(y))$, i.e. at a fiberwise-linear function on $E_1$, as claimed  (whose coefficients $f_{\a}$   depend  non-linearly on the coefficients $g_{\mu}$ of $g$).
\end{proof}

Theorem~\ref{thm.fl} justifies the name ``fiberwise-linear'' for the class of thick morphisms of vector bundles introduced in Definition~\ref{def.fiblingen}: the corresponding pullbacks map functions on the base   to functions on the base  and fiberwise-linear functions    to fiberwise-linear functions. (The latter, in spite of the non-linearity of the pullbacks as such.)

\begin{remark} If for a vector bundle $E_1$, we introduce weights, i.e. a $\ZZ$-grading, by setting $\w(x^a)=0$ for base coordinates and $\w(u^{\a})=+1$ for fiber coordinates, and do the same for $E_2$, and consider the induced grading for the cotangent bundles, i.e. require that the weights of the canonically conjugate variables be  opposite (e.g. $\w(p_a)=-\w(x^a)=0$ and $\w(p_{\a})=-\w(u^{\a})=-1$), then one can characterize generating functions of the form~\eqref{eq.fiblingen} by the condition
\begin{equation}\label{eq.wdegs}
  (\w +\deg)(S)=+1\,,
\end{equation}
where $\deg=\#p_a+\#p_{\a}+\#q_{i}+\#q_{\mu}$ is the degree in the momentum variables. Indeed, for such ``total weight''  $\w +\deg$ we have
\begin{equation*}
  \w +\deg=(\#u^{\a}-\#p_{\a}+\#w^{\mu}-\#q_{\mu})+  (\#p_a+\#p_{\a}+\#q_{i}+\#q_{\mu})= \#p_a+\#u^{\a}+\#q_{i}+\#w^{\mu}\,,
\end{equation*}
and clearly a function $S(x^a,u^{\a};q_i,q_{\mu})$ satisfies~\eqref{eq.wdegs} if and only if has the form~\eqref{eq.fiblingen}.
\end{remark}

\begin{remark} Above, we have considered the case of \emph{even} fiberwise-linear thick morphisms of vector bundles. \emph{Odd}
fiberwise-linear thick morphisms are defined in the same way, and Theorems~\ref{thm.flm} and~\ref{thm.fl} hold mutatis mutandis.
\end{remark}

\section{Thick $Q$-morphisms. Consequences for   forms and cohomology} \label{sec.forms}

In the previous section, we have established that any thick morphism of (super)manifolds, even $M_1\tto M_2$ or odd $M_1\oto M_2$, induces a thick morphism of the antitangent bundles, of the opposite parity: $\Pi TM_1\oto \Pi TM_2$ and $\Pi TM_1\tto \Pi TM_2$, respectively. Since functions on $\Pi TM$ are the same as differential forms on $M$ when $M$ is an ordinary manifold and for a supermanifold, functions on $\Pi TM$ are by definition the Bernstein--Leites  pseudodifferential forms on $M$. 
De Rham differential commutes with ordinary pullbacks of forms (induced by usual smooth maps). What about pullbacks by thick morphisms? In this section, we find the answer. The problem we have to deal with, is the non-linearity of pullbacks. It is useful to look at the case of general $Q$-manifolds.

\subsection{Thick morphisms of $Q$-manifolds} \label{subsec.thickqman}
Recall that a \emph{$Q$-manifold} is a supermanifold equipped with a homological vector field, i.e. an odd vector field $Q$ satisfying $Q^2=0$. (This notion is due to A.~S.~Schwarz~\cite{schwarz:semiclassical}.\footnote{$Q$-manifolds are sometimes referred to as ``DG-manifolds'', but the definition does  not require a $\ZZ$-grading and interesting examples can have none or several, so we prefer the original  terminology.} See also~\cite{schwarz:aksz}) In particular, the algebra  of functions on a $Q$-manifold is a chain complex. (Here by a ``chain complex'' we mean just a super vector space equipped with an odd differential.)

A \emph{morphism of $Q$-manifolds} or, shortly, a \emph{$Q$-morphism}, is a smooth supermanifold map that intertwines the homological vector fields. (As it is known, such are the supergeometric descriptions of the $\Linf$-morphisms of $\Linf$-algebras and the Lie algebroid morphisms.) We shall extend this notion to thick morphisms.

Suppose $M_1$ and $M_2$ are $Q$-manifolds with the homological vector fields $Q_1$ and $Q_2$. In coordinates, $Q_1=Q^a(x)\lder{}{x^a}$ and $Q_1=Q^i(y)\lder{}{y^i}$.

\begin{definition} \label{def.qthick}
An even thick morphism $\F\co M_1\tto M_2$ is a \emph{thick $Q$-morphism} if the Hamiltonians corresponding to $Q_1$ and $Q_2$ are $\F$-related, i.e. are equal on the Lagrangian submanifold in $T^*M_2\times (-T^*M_1)$ given by $\F$. In coordinates,
\begin{equation}\label{eq.qev}
  Q_1^a(x^a)\der{S}{x^a}=Q_2^i\Bigl((-1)^{\itt}\der{S}{q_i}\Bigr)q_i\,,
\end{equation}
where $S=S(x^a;q_i)$ is the (even) generating function of $\F$.

Likewise, an odd thick morphism $\Psi\co M_1\oto M_2$ is a \emph{thick $Q$-morphism} if the multivector fields corresponding to $Q_1$ and $Q_2$ are $\Psi$-related, i.e. are equal on the Lagrangian submanifold in $\Pi T^*M_2\times (-\Pi T^*M_1)$  given by $\Psi$. In coordinates,
\begin{equation}\label{eq.qod}
  Q_1^a(x^a)\der{S}{x^a}=Q_2^i\Bigl(\der{S}{y^*_i}\Bigr)y^*_i\,,
\end{equation}
where $S=S(x^a;y^*_i)$ is the (odd) generating function of $\Psi$.
\end{definition}

Definition~\ref{def.qthick} can be seen as a special case of the notions of $\Sinf$- and $\Pinf$-thick morphisms introduced in~\cite{tv:nonlinearpullback} if one notes that a $Q$-structure can be regarded both as a $\Sinf$-structure and a $\Pinf$-structure with a single odd unary bracket given by $Q$ as an operator on functions.

For an ordinary $Q$-morphism $\f\co M_1\to M_2$,    the pullback $\f^*\co \fun(M_2)\to \fun(M_1)$   is by the definition a chain map (w.r.t. the differentials $Q_1$ and $Q_2$). What about pullbacks by thick $Q$-morphisms? Anticipating the answer, since these pullbacks are in general non-linear, we need to explain what would be a ``non-linear chain map'' of complexes. Actually, this is nothing but an $\Linf$-morphism of complexes, which can be elaborated as follows.

Let $V$ and $W$ be complexes, i.e. (super) vector spaces equipped with odd operators $d_1$ and $d_2$ of square zero. We can consider them as supermanifolds 
and the linear operators $d_1$ and $d_2$, as   vector fields, which will be homological. By   parity reversion, the same works for $\Pi V$ and $\Pi W$ (with the operators $d_1^{\Pi}$ and $d_2^{\Pi}$, which for simplicity of notation, we will denote simply by $d_1$ and $d_2$).

\begin{definition} \label{def.nlchain}
A \emph{non-linear chain map} $f\co V\to W$ is a formal map  $f\co  {V}\to  {W}$, with $V$ and $W$ regarded as supermanifolds\footnote{with zero  as a  base point}, which is a $Q$-morphism w.r.t. $d_1$ and $d_2$.
\end{definition}

In the same way we can speak about non-linear chain maps $f\co \Pi V\to \Pi W$.

(With an abuse of language, we can refer to the corresponding formal supermanifold maps $f\co  {V}\to  {W}$ or $f\co \Pi V\to \Pi W$ also as non-linear chain maps.)

If one expands a non-linear chain map $f\co V\to W$ into the Taylor series, its terms will give even symmetric multilinear maps of super vector spaces
\begin{equation*}
  f_k\co V\times \ldots \times V\to W\,,
\end{equation*}
$k=0,1,2,\ldots $ (equivalently, even linear maps $S^kV\to W$) satisfying the   relations
\begin{align*}
  d(f_0) & =0\,, \\
  d(f_1(v)) & =f_1(d(v))\,,\\
  d(f_2(v_1,v_2)) & =f_2(d(v_1),v_2)+(-1)^{\vt_1}f_2(v_1,d(v_2))\,,\\
  d(f_3(v_1,v_2,v_3)) & =f_3(d(v_1),v_2,v_3)+(-1)^{\vt_1}f_3(v_1,d(v_2),v_3)+(-1)^{\vt_1+\vt_2}f_3(v_1,v_2,d(v_3))\,,\\
  \dots
\end{align*}
where we have put $d$ for both  $d_1$ and $d_2$; i.e. the element $f_0\in W$ is closed, the linear map $f_1\co V\to W$ is a usual chain map, and the Leibniz rule is satisfied for each 
$f_k$, $k\geq 2$.
These are just the relations for an $\Linf$-morphism specialized for our case (see e.g.~\cite{tv:gradedmicro}).

In the same way, a non-linear chain map $f\co \Pi V\to \Pi W$ expands into the Taylor terms, equivalent to a sequence of antisymmetric multilinear maps from $V$ to $W$ of alternating parities (or  linear maps $\Lambda^kV\to W$ of parities $k-1$) satisfying  relations similar to   above.

The chain map $f_1\co V\to W$ induces as usual an  even  linear map  $f_{1*}\co H(V,d_1)\to H(W,d_2)$. Together with it, a standard argument based on the Leibniz rule applied to ``higher'' $f_k$ gives the following statement:
\begin{proposition}
A non-linear chain map $f\co V\to W$ induces a sequence of even symmetric multilinear maps of cohomology
\begin{equation*}
  f_{k*}\co H(V,d_1)\times \ldots \times H(V,d_1)\to H(W,d_2)\,,
\end{equation*}
$k=0,1,2,\ldots $, which can be assembled into a single formal map
\begin{equation}
  f_{*}\co H(V,d_1)\to H(W,d_2)\,,
\end{equation}
of the cohomology spaces regarded as supermanifolds, where $0$ is mapped to $[f_0]\in H(W,d_2)$.
\end{proposition}

Similar statement holds for non-linear chain maps $\Pi V\to \Pi W$.

Now we can return to thick $Q$-morphisms.

\begin{theorem}
  Let $\F\co M_1\tto M_2$ be an even thick $Q$-morphism. Then the pullback
  \begin{equation*}
    \F^*\co \funn(M_2)\to \funn(M_1)
  \end{equation*}
  is a non-linear chain map.

  Similarly, let $\Psi\co M_1\oto M_2$ be an odd  thick $Q$-morphism. Then the pullback
  \begin{equation*}
    \Psi^*\co \ofunn(M_2)\to \ofunn(M_1)
  \end{equation*}
is a non-linear chain map.
\end{theorem}
\begin{proof}
  Consider the case of an even thick morphism $\F\co M_1\tto M_2$. (The odd case can be proved in the same way.) Suppose the generating function of $\F$ satisfies~\eqref{eq.qev}. We need to show that the pullback $\F^*\co \funn(M_2)\to \funn(M_1)$ is a non-linear chain map with respect to the $Q_2$ and $Q_1$ as the differentials. According to Definition~\ref{def.nlchain} this would mean the pullback $\F^*$ intertwining the corresponding linear vector fields $\hat Q_2\in \Vect(\funn(M_2))$ and $\hat Q_1\in \Vect(\funn(M_1))$ on the spaces of functions. This is a special case of the statement for $\Sinf$-structure (see~Theorem~6 in~\cite{tv:nonlinearpullback} or  Theorem~4.4 in~\cite{tv:gradedmicro}). For completeness, we provide the argument, which is as follows. We need to show that
  \begin{equation}\label{eq.inttw}
    (\id +\e \hat Q_1)\circ \F^*=\F^*\circ (\id +\e \hat Q_2)
  \end{equation}
  where $\e^2=0$. Note that, for $Q_1$, $(\id +\e \hat Q_1)f(x)=f(x)+\e Q^a(x)\der{f}{x^a}(x)$, and the same for $Q_2$. So if we start from some $g\in \funn(M_2)$, we have $\F^*[g]=f$, where
   \begin{equation*}
     f(x)=S(x,q)+g(y)-y^iq_i
   \end{equation*}
   and $y$ and $q$ are defined from the equations
  \begin{align*}
    q_i & =\der{g}{y}(y)\,, \\
    y^i & =(-1)^{\itt}\der{S}{q_i}(x,q)\,.
  \end{align*}
  We have
  \begin{multline*}
    \der{f}{x^a}(x)=\der{S}{x^a}+ \der{q_i}{x^a}\der{S}{q_i}+\der{y^i}{x^a}\der{g}{y^i}-\der{y^i}{x^a}q_i-(-1)^{\itt}\der{q_i}{x^a}y^i=
    \der{S}{x^a}\Bigl(x,\der{g}{y}(y)\Bigl)\,,
  \end{multline*}
  where $y$ is found from
  \begin{equation}\label{eq.phig}
    y^i=(-1)^{\itt}\der{S}{q_i}\Bigl(x,\der{g}{y}(y)\Bigl)\,.
  \end{equation}
  Hence
  \begin{equation}\label{eq.comp1}
    \bigl((\id +\e \hat Q_1)\circ \F^*\bigr)[g]= f(x)+\e Q^a(x)\der{S}{x^a}\Bigl(x,\der{g}{y}(y)\Bigl)\,.
  \end{equation}
  On the other hand, we know that the derivative of $\F^*$ at $g\in \funn(M_2)$ is the usual pullback $\f^*_g$, i.e. the substitution   $y=y(x)$ defined by~\eqref{eq.phig} (by Theorem~2 in~\cite{tv:nonlinearpullback}). Hence
  \begin{equation}\label{eq.comp2}
    \bigl(\F^*\circ (\id +\e \hat Q_2)\bigr)[g]=\F^*\bigl[g+\e \hat Q_2 g\bigr]=\F^*[g]+\e \f^*_g\bigl(\hat Q_2 g\bigr)
    =f(x)+\e Q^i(y)\der{g}{y^i}(y)\,,
  \end{equation}
  where $y$ is determined from~\eqref{eq.phig}. From~\eqref{eq.qev} and~\eqref{eq.phig}, we obtain
  \begin{equation*}
    Q^a(x)\der{S}{x^a}\Bigl(x,\der{g}{y}(y)\Bigl)=Q^i(y)\der{g}{y^i}(y)\,,
  \end{equation*}
  hence the right-hand sides of~\eqref{eq.comp1} and~\eqref{eq.comp2} coincide, and we have the equality~\eqref{eq.inttw} as claimed.
\end{proof}

\begin{corollary}
An even thick $Q$-morphism $\F\co M_1\tto M_2$ induces a non-linear formal map on cohomology
\begin{equation}
  \F^*\co H(\funn(M_2),Q_2)\to H(\funn(M_1),Q_1)\,
\end{equation}
regarded as supermanifolds. Similarly, an odd
thick $Q$-morphism $\Psi\co M_1\oto M_2$ induces a non-linear formal map on cohomology with reversed parity
\begin{equation}
  \F^*\co \Pi H(\funn(M_2),Q_2)\to \Pi H(\funn(M_1),Q_1)\,
\end{equation}
regarded as supermanifolds.
\end{corollary}

\subsection{Action of thick morphisms on forms and de Rham cohomology} \label{subsec.thickderham}

Let us return to the setup of Section~\ref{sec.tang}. Consider supermanifolds $M_1$ and $M_2$ and a thick morphism $\F\co M_1\tto M_2$. (The case of odd thick morphisms $M_1\oto M_2$ can be treated similarly.)

From Section~\ref{sec.tang}, we know that there is the induced antitangent odd thick morphism
\begin{equation}\label{eq.atan}
  \Pi T\F\co \Pi TM_1\oto \Pi TM_2
\end{equation}
and hence we have the non-linear pullback on odd forms
\begin{equation}\label{eq.phistarforms}
  \F^*\co \pomm(M_2)\to \pomm(M_1)\,.
\end{equation}
Here by ``forms'' we mean pseudodifferential forms and use the notation $\O(M):=\fun(\Pi TM)$). We use boldface to denote the corresponding infinite-dimensional supermanifolds, such as $\omm(M)$ (whose points are the even forms) and $\pomm(M)$ (whose points are the odd forms). In~\eqref{eq.phistarforms}, we write  $\F^*$ for $(\Pi T\F)^*$, the pullback by~\eqref{eq.atan}.

\begin{theorem}
  The antitangent   morphism~\eqref{eq.atan} is a thick $Q$-morphism (with respect to the de Rham differentials).
\end{theorem}
\begin{proof}
Let   $S=S(x;q)$ be the generating function for $\F$.
In the description of $\Pi T\F$ we use the identification
\begin{equation*}
  \Pi T^*(\Pi TM)\cong \Pi T(T^*M)
\end{equation*}
established in Section~\ref{sec.tang}. We  change back the notation  and use $d$  instead of $\p$  because here we will not need forms on the antitangent bundles. In order to prove that $\Pi T\F$ is a thick $Q$-morphism, we need to show that multivector fields corresponding to   $d$ on $M_1$ and   $M_2$,  i.e. $dx^ax^*_a$ and $dy^iy^*_i$, are equal on $\Pi T\F$ (see~\eqref{eq.qod} for the abstract case). Taking into account formulas~\eqref{eq.oddtau3}, we need to show that
\begin{equation}\label{eq.d}
  dx^a(-1)^{\at+1}dp_a=dy^i(-1)^{\itt+1}dq_i\,,
\end{equation}
if we have the equations for $p_a$ and $y^i$\,:
\begin{align*}
  p_a & =\der{S}{x^a}\,, \\
  y^i & =(-1)^{\itt}\der{S}{q_i}
\end{align*}
and the corresponding equations for $dp_a$ and $dy^i$ obtained by applying $d$\,:
\begin{align*}
  dp_a & =dx^b\dder{S}{x^b}{x^a} + dq_i\dder{S}{q_i}{x^a}\,,\\
  dy^i & =(-1)^{\itt}\left(dx^a\dder{S}{x^a}{q_i} +dq_j\dder{S}{q_j}{q_i}\right)\,.
\end{align*}
Hence we have for the left-hand side of~\eqref{eq.d}\,:
\begin{equation}\label{eq.dxdp}
  (-1)^{\at+1}dx^adp_a=(-1)^{\at+1}dx^adx^b\dder{S}{x^b}{x^a}+(-1)^{\at+1}dx^adq_i\dder{S}{q_i}{x^a}
\end{equation}
and for the right-hand side of~\eqref{eq.d}\,:
\begin{equation}\label{eq.dydq}
  (-1)^{\itt+1}dy^idq_i=-dx^a\dder{S}{x^a}{q_i}dq_i -dq_j\dder{S}{q_j}{q_i}dq_i\,.
\end{equation}
The terms with $dx^adx^b$  and with $dq_idq_j$ in~\eqref{eq.dxdp} and~\eqref{eq.dydq}  identically vanish and we can see that  the remaining terms with $dx^adq_i$ are equal. Hence~\eqref{eq.d} holds, and the theorem is proved.
\end{proof}

\begin{corollary}
The pullback on odd forms~\eqref{eq.phistarforms} is a non-linear chain map, inducing a non-linear formal map
\begin{equation}\label{eq.phistarcoh}
  \Psi^*\co \PH^*(M_2)\to \PH^*(M_1)\,.
\end{equation}
\end{corollary}
Here we have used boldface to emphasize that we consider vector supermanifolds corresponding to the de Rham cohomology spaces. (Recall that the cohomology of pseudodifferential forms is isomorphic to the ordinary cohomology of the underlying topological space, see e.g.~\cite{tv:gitnew}.)

Mutatis mutandis, the same holds for  an odd thick morphism $\Psi\co M_1\oto M_2$. The antitangent morphism
\begin{equation}\label{eq.aotan}
  \Pi T\Psi\co \Pi TM_1\tto \Pi TM_2
\end{equation}
will be a thick $Q$-morphism, and hence the pullback on even forms,
\begin{equation}\label{eq.psistarforms}
  \Psi^*\co \omm(M_2)\to \omm(M_1)\,,
\end{equation}
will be a non-linear chain map, inducing a non-linear formal map of de Rham cohomology:
\begin{equation}\label{eq.psistarcoh}
  \Psi^*\co \HH^*(M_2)\to \HH^*(M_1)\,.
\end{equation}

Further work is required to see if the higher terms of the maps~\eqref{eq.phistarcoh} and~\eqref{eq.psistarcoh} can be non-trivial, and  explore significance of such non-linear cohomology transformations.




\def\cprime{$'$}

\end{document}